\documentclass[11pt]{article}

\usepackage{amsmath,amssymb,amsthm,color,graphicx}
\usepackage[latin1]{inputenc}
\def\Nset{\mathbb N}

\def\Rset{\mathbb R}
\def\Zset{\mathbb Z}

\def\ind{\mathbf{1}}
\def\esp{\mathbb E}
\def\pr{\mathbb P}
\def\var{\mathrm{var}}
\def\cov{\mathrm{cov}}
\def\re{\mathrm{Re}}
\def\im{\mathrm{Im}}
\def\rme{\mathrm{e}}
\def\rmi{\mathrm{i}}

 \def\mcd{\mathcal{D}}
 \def\mcf{\mathcal{F}}
 \def\mcm{{\mathcal M}}

\newtheorem{theo}{Theorem}[section]
\newtheorem{lem}{Lemma}[section]

\newtheorem{prop}{Proposition}[section]
\theoremstyle{remark}
\newtheorem{rem}{Remark}[section]

\numberwithin{equation}{section}
\newcommand{\ie}{{\em i.e.} }

\begin{document}

\title{Asymptotics for Duration-Driven Long Range Dependent Processes}
\author{ Mengchen Hsieh\thanks{New York University} \and Clifford M.
  Hurvich$^*$ \and Philippe Soulier\thanks{Universit\'e Paris X
  \newline We thank the Editor and two referees for their careful
  reading of the paper and for their constructive comments.}}
\maketitle

\begin{abstract}
We consider processes with second order long range dependence
resulting from heavy tailed durations. We refer to this phenomenon
as duration-driven long range dependence (DDLRD), as opposed to
the more widely studied linear long range dependence based on
fractional differencing of an $iid$ process. We consider in detail
two specific processes having DDLRD, originally presented in Taqqu
and Levy (1986), and Parke (1999). For these processes, we obtain
the limiting distribution of suitably standardized discrete
Fourier transforms (DFTs) and sample autocovariances. At low
frequencies, the standardized DFTs converge to a stable law, as do
the standardized sample autocovariances at fixed lags. Finite
collections of standardized sample autocovariances at a fixed set
of lags converge to a degenerate distribution. The standardized
DFTs at high frequencies converge to a Gaussian law. Our
asymptotic results are strikingly similar for the two DDLRD
processes studied. We calibrate our asymptotic results with a
simulation study which also investigates the properties of the
semiparametric log periodogram regression estimator of the memory
parameter.
\end{abstract}



\section{Introduction}

The renewal-reward process of Taqqu and Levy (1986) and the error
duration model of Parke (1999) are nonlinear models with long
memory.  Both models embody useful features not shared by
traditional linear long-memory models (such as ARFIMA), in that
they both allow simultaneously for regime switching and long
memory without requiring these two commonly-observed phenomena to
be separately parameterized. The models intrinsically possess both
structural change and long memory, in an inextricably intertwined
manner, and thus may help practitioners to view these two
phenomena as a duality rather than a dichotomy.

In the renewal-reward process, the value of the process stays constant
at some random level throughout regimes of durations governed by a
sequence of i.i.d. random variables with finite mean but infinite
variance. In the model of Parke (1999), the process is written as a
sum of present and past shocks, where shocks survive in the sum for
durations governed by a long-tailed i.i.d. sequence. In both models,
there is a one-to-one correspondence between the tail index of the
i.i.d. duration sequence and the memory parameter of the process.
Therefore, we will say that both of these models possess
\textit{duration-driven long range dependence} (DDLRD).

Liu (2000) has used the renewal-reward process in a stochastic
volatility model for financial returns exhibiting simultaneous
long memory and regime switching in the volatility. The economic
motivation for such a model is a scenario where trading is
stimulated by news arrivals, and the duration of a volatility
regime created by a given news event is heavy-tailed. Liu (2000)
fitted a such a stochastic volatility model to the returns on the
S\&P 500 and found that it was successful at simultaneously
capturing the long memory and heavy-tailed regime switching of
volatility, and that it was successful at forecasting volatility.

The error duration model of Parke (1999) has drawn considerable
recent attention among practitioners in finance and economics. By
focusing on the duration of shocks rather than on fractional
differencing of the shocks, the model provides an appealing
paradigm for long memory in economic time series and in volatility
of financial series. For example, Bollerslev and Jubinski (1999)
invoked Parke's error duration mechanism to argue that under
certain reasonable assumptions on the duration of the impact of
particular news events, the aggregate information arrival process
will have long memory, a conclusion that supports a version of the
Mixture of Distributions Hypothesis (MDH). Another relevant
example was given by Parke (1999), who argued that the error
duration model provides a plausible mechanism for explaining long
memory in aggregate employment. He assumed that an error
represents the effect of a given firm on aggregate employment and
holds constant for the lifetime of the firm. He then analyzed
survival rates for U.S. businesses, and showed that the rates were
consistent with an error duration model that induces long memory
in employment.

As models possessing DDLRD gain increasing application,
practitioners may feel that, if faced with data generated by a
model having DDLRD, they can safely use the standard methods of
data analysis and statistical inference for long-memory series. In
particular, they may wish to examine the sample autocovariances,
or to construct the log-periodogram regression estimator (GPH) of
the memory parameter, due to Geweke and Porter-Hudak (1983), or to
use the Gaussian semiparametric estimator (GSE) of K\"unsch
(1987). Some caution may be in order here, however, since most of
the existing theory assumes that the series is either Gaussian
(see Robinson 1995a and Hurvich, Deo and Brodsky 1998 for GPH),
linear in an i.i.d. sequence (see Velasco 2000 for GPH), linear in
a Martingale difference sequence (See Robinson 1995b for GSE,
Chung 2002 for autocovariances), or, in the case of volatility,
that the observations can be transformed into a sum of linear
series (see Deo and Hurvich 2001, Hurvich and Soulier 2002, for
GPH applied to long memory stochastic volatility models). If the
Taqqu-Levy and Parke models are to be widely accepted and used, it
is necessary to build a theory for the currently-standard
methodology of long-memory data analysis and inference that
applies to such series. The present paper represents a first step
in that direction. We will explore the asymptotic properties of
the discrete Fourier transforms (DFTs) and sample autocovariances
from both processes. Some of the results are surprising, and tend
to confirm that caution was indeed warranted.

One surprising result we find is that both the sample
autocovariances at a fixed lag and the DFT at a fixed Fourier
frequency, if suitably standardized, have limiting non-Gaussian
stable distributions. This implies that a data analysis based on
examination of the sample autocovariances may be misleading. It
also implies that data analytic methods that rely on the very low
frequency behavior of the DFT of a series with DDLRD will not have
the same asymptotic properties as in the linear long-memory case.
(See, e.g., Chen and Hurvich (2003 a,b) on fractional
cointegration of linear processes). On a more positive note, but
still surprisingly, we find for the DFT at the $j$'th Fourier
frequency $x_j=2\pi j/n$ where $n$ is the sample size, that if $j$
tends to $\infty$ sufficiently quickly, then the DFT is
asymptotically normal. This indicates that the DFT at not-too-low
frequencies has some robustness to the type of long-memory
generating mechanism. It also suggests that standard estimation
methods such as GPH and GSE may retain the same properties that
they are already known to have in the linear case, although some
trimming of very low frequencies may be needed. Our theoretical
results will be augmented with a Monte Carlo study, both to
calibrate the finite-sample applicability of our theorems, and to
briefly explore the properties of the GPH estimator for models
with DDLRD, a topic which we do not attempt to handle
theoretically here.

The organization of the remainder of this paper is as follows. In
Section 2, we review some of the existing theory on second order
long memory processes, so as to contrast it with the theory we
will develop for the DDLRD processes. In Section 3, we give the
precise formulation of the Taqqu-Levy and Parke models, exhibit
their autocovariance functions, present a proposition which shows
that Parke's process is well defined only in the stationary case,
elaborate further on the differences between the models, and
present some basic theory for the two models. In particular, in
Section 3.4 we consider the weak convergence of partial sums for
both processes, and in Section 3.5 we consider asymptotics for the
empirical process in the Taqqu-Levy case. In Section 4, we present
the asymptotics for the discrete Fourier transforms for both
series, treating the cases of low frequencies and high frequencies
separately, as the limiting distribution is different in these two
cases. In Section 5, we consider the asymptotics for the sample
autocovariances of the Parke and Taqqu-Levy processes.
Interestingly, the joint limiting distribution of a collection of
standardized sample autocovariances at a fixed finite set of lags
is degenerate. In Section 6 we present the results of a simulation
study. In Section 7, we present some concluding remarks. In the
Appendix, we present some useful lemmas, and give the proofs of
our main results.

\section{Second order long memory}
\label{sec:rappels} We start by recalling some classical
definitions and facts about long memory processes.  A second order
stationary process $X=\{X_t,\ t\in\Zset\}$ is usually said to be
long range dependent if its autocovariance function
$\gamma(t)=\cov(X_0,X_t)$ is not absolutely summable. This
definition is too wide to be useful. A more practical condition is
that the autocovariance is regularly varying: there exist
$H\in(1/2,1)$ and a function $L$, slowly varying function at $\infty$, such that
\begin{equation}
  \label{eq:deflrd}
  \gamma(t) = L(t) |t|^{2H-2}.
\end{equation}
A function $L$ is slowly varying at $\infty$ if it is bounded on
finite intervals and if $L(at)/L(t) \rightarrow 1$ as
$t\rightarrow \infty$ for all $a>0$. For fractional Gaussian noise
(i.e., the increments of a fractional Brownian motion) $L(t)$ is a
positive constant. For the $ARFIMA(p,d,q)$ model of Granger and
Joyeux (1980), Hosking (1981), $L(t)$ approaches a positive
constant as $t \rightarrow \infty$. Other examples of functions
$L$ that are slowly varying at $\infty$ include $\log t$, powers
of $\log t$, and iterated logarithms. For more details, see
Resnick (1987) or Bingham, Goldie and Teugels (1989). Since the
autocovariances of the processes considered in this paper are
generated by tail probabilities of positive random variables, our
general regular variation assumption is needed to allow our theory
to cover cases of practical interest.

Under condition (\ref{eq:deflrd}), it holds that:
\begin{gather}  \label{eq:varsum}
     \lim_{n\to\infty} n^{-2H} L(n)^{-1}   \var \left(
     \sum_{t=1}^n X_t \right) = -4 \, \Gamma(-2H) \cos(\pi H).
\end{gather}
A second order stationary process satisfying \eqref{eq:varsum}
will be referred to as a second order long memory process, and the
coefficient $H$ is the long memory parameter of the process, often
referred to as the Hurst coefficient of the process $X$. We will
henceforth use this terminology.

A weakly stationary process with autocovariance function
satisfying (\ref{eq:deflrd}) has a spectral density, \ie there
exists a function $f$ such that
\[
\gamma(t) = \int_{-\pi}^\pi f(x) \rme^{\rmi tx} dx.
\]
The function $f$ is the sum of the series
\[
\frac1{2\pi} \sum_{t\in \Zset} \gamma(t) \rme^{-\rmi tx},
\]
which converges uniformly on the compact subsets of
$[-\pi,\pi]\setminus\{0\}$ and in $L^1([-\pi,\pi],dx)$. It is then
well known that the behaviour of the function $f$ at zero is
related to the rate of decay of~$\gamma$. More precisely, if we
assume in addition that $L$ is ultimately monotone, we obtain the
following Tauberian result:
\begin{gather} \label{eq:spectraldensity}
  \lim_{x \to 0} L(1/x)^{-1} x^{2H-1} f(x) = \pi^{-1} \Gamma(2H-1)
  \sin(\pi H).
\end{gather}
(Cf. for instance Taqqu (2003), Proposition 4.1). The usual tools
of statistical analysis of weakly stationary processes are the
empirical autocovariance function, the discrete Fourier transform
(DFT) and the periodogram. We will focus here on the DFT and
periodogram ordinates of a sample $X_1,\dots,X_n$, defined as
\[
d_{X,k} = (2\pi n)^{-1/2} \sum_{t=1}^n X_t \rme^{\rmi t x_k}, \ \
I_{X,k} = |d_{X,k}|^2,
\]
for integers $k$, $1 \leq k < n/2$. In the classical weakly
stationary short memory case (when the autocovariance function is
absolutely summable), it is well known that the periodogram is an
asymptotically unbiased estimator of the spectral density.  This
is no longer true for second order long memory processes. Hurvich
and Beltrao (1993) showed (assuming that $L(1/x)$ is continuous at
$x=0$, though the extension is straightforward) that for any fixed
positive integer $k$, there exists a constant $c(k,H)>0$ such that
\[
\lim_{n\to\infty} \esp[I_{X,k}/f(x_k)] = c(k,H).
\]
The previous results hold for any second order long memory
process.  We now describe some weak convergence results that are
valid for Gaussian or linear processes.

If $X$ is a second order long memory Gaussian process, then
$L(n)^{-1/2}n^{-H} \sum_{k=1}^{[nt]} X_k$ converges weakly to the
fractional Brownian motion $B_H(t)$ which is the zero mean
Gaussian process with covariance function given by:
\[
\esp[B_H(s)B_H(t)] = \frac12
\left(|s|^{2H}-|t-s|^{2H}+|t|^{2H}\right).
\]
Here weak convergence is in the space $\mcd$ of right-continuous
and left-limited (c\`adl\`ag) functions on $[0,\infty)$.

This result can be extended to a strict sense linear process, \ie
a process $X$ for which there exists a sequence
$(\epsilon_j)_{j\in\Zset}$ of i.i.d. random variables with zero
mean and finite variance, and a square summable sequence of real
numbers $(a_j)_{j\in\Zset}$ such that for all $t\in\Zset$,
\[
X_t = \sum_{j\in\Zset} a_j \epsilon_{t-j}.
\]
If $a_j=L(j) |j|^{H-3/2}$, then $X$ is a second order long memory
process with Hurst coefficient $H$, and $L(n)^{-1/2} n^{-H}$
$\sum_{k=1}^{[nt]} X_k$ converges weakly, in the sense of weak
convergence of finite dimensional distributions, to the fractional
Brownian motion $B_H(t)$.  This can be proved easily by applying
the Central Limit Theorem for linear processes of Ibragimov and
Linnick (1971, Theorem 18.6.4).  \nocite{ibragimov:linnik:1971}
Weak convergence in the space
$\mcd$ can also be proved. Cf. 
Gorodeckii (1977) or  Lang and
Soulier (2000).  
\nocite{gorodeckii:1977} \nocite{lang:soulier:2000} A classical
example of such a long memory linear process is the
ARFIMA($p,d,q$) process, whose Hurst coefficient is $H=1/2+d$.

For Gaussian and linear processes, a weak convergence result can
also be obtained for the periodogram and the DFT ordinates.  For
any fixed $j$, $f(x_j)^{-1/2} d_{X,j}$ converges to a complex
Gaussian distribution with dependent real and imaginary parts.
Cf. Terrin and Hurvich (1994),  Chen and Hurvich (2003 a,b),
Walker (2000), and Lahiri (2003).

The asymptotic behaviour described above is different from the
behaviour of weakly dependent processes, such as sequences of
i.i.d.  or strongly mixing random variables, whose partial sum
process, renormalised by the usual rate $\sqrt n$, converges to
the standard Brownian motion.  But these long memory processes
share with weakly dependent processes the Gaussian limit and the
fact that weak limits and $L^2$ limits have consistent
normalisations, in the sense that, if $\xi_n$ denotes one of the
statistics considered above, there exists a sequence $v_n$ such
that $v_n \xi_n$ converges weakly to a non degenerate distribution
and $v_n^2 \esp[\xi_n^2]$ converges to a positive limit (which is
the variance of the asymptotic distribution).

In the sequel we define two second order stationary models, which
possess properties (\ref{eq:deflrd}) and
(\ref{eq:spectraldensity}), but whose weak limit behaviour is
extremely different from that of Gaussian or linear models. In
Section \ref{sec:formulation} we define these models.
\section{Formulation of the Models} \label{sec:formulation}
\subsection{The Taqqu-Levy Model}
Let $\{T_k\}$ be i.i.d. positive integer-valued random variables
with mean $\mu$, in the domain of attraction of a stable
distribution with tail index $\alpha \in (1,2)$, \ie there exists
a function $L$, slowly varying at infinity such that for all
$n\geq1$,
\begin{align} \label{eq:regvarrenewal}
 \pr(T_1 \geq n) =L(n) n^{-\alpha}.
\end{align}
To avoid trivialities, we also assume that $\pr(T_1=1)>0$.  Let
$S_0$ be a non-negative integer-valued random variable,
independent of the $\{T_k\}$, with probability distribution
\begin{equation} \label{eq:SoEq}
P(S_0=u)=\mu^{-1} P(T_k \geq u+1), \ u=0,1,\ldots.
\end{equation}
Let $\{W_k\}$ be i.i.d. random variables with $E[W_k]=0$ and
$\var[W_k]=\sigma_W^2 < \infty$.  Assume that the $\{W_k\}$ are
independent of $S_0$ and $\{T_k\}$. We observe a process denoted
by $\{X_t\}$ for $t=1,\ldots , n$. The observed process is
constant on regimes (intervals) determined by $S_0$ and the
interarrival times $T_k$. The constant value on each regime is
given by one of the $\{W_k\}$. The time between the start of the
sample and the first change of regime is $S_0$, and the subsequent
waiting times are $T_1,T_2,\ldots$. The total time up to the end
of the $k$'th regime $(k=0,1,\ldots)$ is given by $S_{-1} \equiv
-1$, $S_0$ and
\[
S_k=S_0+T_1+\ldots+T_k,  \ k=1,2,\ldots.
\]
The observed process $\{X_t\}$ is given by $W_k$ if $t$ lies in
the $k$'th regime, so that
\begin{equation} \label{eq:defx}
X_t=\sum_{k=0}^{\infty} W_k \ind_{\{S_{k-1} \leq t < S_k\}},
\end{equation}
where $\ind_A$ is the indicator function of the set $A$.  Let
$M_n$ be the counting process associated with the renewal process
$\{S_0,S_1,\dots\}$, \ie a non-negative integer-valued random
variable denoting the total number of regime changes in the series
before the time $n-1$:
\[
M_n = k \Leftrightarrow S_{k-1} \leq n < S_k.
\]
The renewal process $\{S_0,S_1,\dots\}$ is called a stationary
renewal process, in the sense that the counting process $M_n$ has
stationary increments, whence the following result (Cf. Liu 2000,
Theorem 1.1).
\begin{prop}  \label{prop:Xstationnaire}
  The process $X$ defined by \eqref{eq:defx} is strictly stationary
  with zero mean and covariances
  \begin{gather*}
    \cov(X_0,X_r) = \sigma_W^2 \pr(S_0\geq r) = \mu^{-1} \sigma_W^2
    \esp[(T_1-r) \ind_{\{T_1 \geq r\}}].
  \end{gather*}
If (\ref{eq:regvarrenewal}) holds with $1 < \alpha < 2$ and $L$
ultimately monotone, then $X$ is second order long memory with
  Hurst coefficient $H=(3-\alpha)/2$ and spectral density $f$
  satisfying
\begin{gather*}
  \lim_{x\to0} L(1/x)^{-1} x^{2H-1} f(x) =
  \frac{\sigma_W^2 }{2\pi(1-H) \mu} \; \Gamma(2H-1) \sin(\pi H).
  \end{gather*}
\end{prop}
\subsection{The Parke Model}
Let $(\epsilon_t)_{t\in\Zset}$ be a sequence of i.i.d. random
variables with zero mean and $(n_s)_{s\in\Zset}$ be a sequence of i.i.d.
non-negative integer valued random variables which is independent
of $(\epsilon_t)_{t\in\Zset}$.  For $s\in\Zset$, define
\[
g_{s,t}=1  \Leftrightarrow  s \leq t \leq s + n_s.
\]
Parke's error duration process is then defined as:
\[
X_t = \sum_{s\leq t} g_{s,t} \epsilon_s,
\]
Let $N$ be a generic random variable with the same distribution as
the $n_s$, and define
\[
p_k := \pr(N \geq k) \ \ k \geq 0.
\]
$(p_k)_{k\geq 0}$ is then a non-increasing sequence such that
$p_0=1$ and $\lim_{k\to\infty} p_k = 0$.

Parke (1999) does not discuss the existence of this process. In
his main result, he assumes that it is well defined and second
order stationary. Since the terms in the sum defining the process
are not vanishing, by well defined we mean that the sum is almost
surely finite.  We now give a necessary and sufficient condition
for the process $X$ to be well defined.
\begin{prop} \label{prop:WellDefined}
  Parke's process is well defined if and only if $\esp[N] < \infty$.
  In that case it is strictly stationary. If moreover $\epsilon_0$ has
  variance $\sigma_\epsilon^2$, then Parke's
  process has mean zero, finite variance and
  covariances
\begin{gather*}
  \cov(X_0,X_r) = \sigma_\epsilon^2 \sum_{j \geq r} p_j =
  \sigma_\epsilon^2 \, \esp[(N+1-r) \ind_{\{N \geq r\}}]  =
  \sigma_\epsilon^2\sum_{k=r}^\infty p_k.
\end{gather*}
If the survival probabilities $p_k$ are regularly varying with
index $\alpha\in(1,2)$, \ie if they satisfy
\begin{gather}  \label{eq:regvar}
  p_j = \pr(N \geq j) = L(j) j^{-\alpha}, \ \ \ j \geq 1,
\end{gather}
where $L$ is slowly varying and ultimately monotone at infinity,
then Parke's error duration model $X$ exhibits second order long
memory with Hurst coefficient $H= (3-\alpha)/2$ and its spectral
density $f$ satisfies
  \begin{gather*}
    \lim_{x\to0} L(1/x)^{-1} x^{2H-1}f(x) =
    \frac{\sigma_\epsilon^2 }{2\pi(1-H)} \; \Gamma(2H-1) \sin(\pi H).
  \end{gather*}
\end{prop}

\subsection{Differences between the two models}

In the error duration process, the durations are independent of
birth dates which are deterministic and shocks can overlap; in the
renewal-reward process, the durations are exactly equal to the
interval between two consecutive birth dates, which form a renewal
process, so that there is only one surviving reward at any given
time point, and it is precisely the value of the process. In both
cases long memory is caused by the heavy-tailedness of the
durations. This property implies that some of the durations are
eventually extremely long, as illustrated in
Figures~\ref{fig:parke} and~\ref{fig:renewal}.

Variations on these processes, which retain their main features,
are possible; see Deo, Hsieh, Hurvich and Soulier (2006) for a
detailed account.  Given the econometric motivation of the error
duration and renewal-reward processes given by Parke (1999) and
Liu (2000), and since it would be difficult to present a
completely unified presentation and proof of our results, we
consider in this paper only these two processes.

In Figure~\ref{fig:parke}, the horizontal lines represent the
durations $n_s$ of the shocks $\epsilon_s$. The value of the error
duration process (the bullet) at some time point is the sum of
present and past shocks. The shock $\epsilon_{t-2}$ lasts for a
very long time and is still present at time $t+1$, whereas the
shock $\epsilon_{t-1}$ exists only over one period of time.

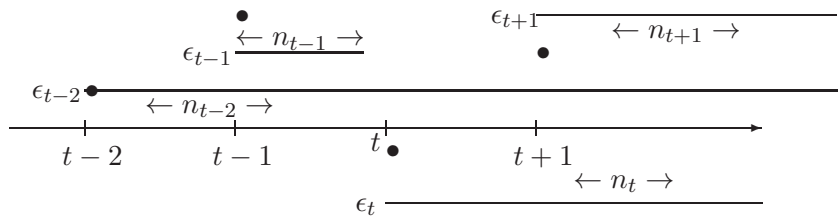
\begin{figure}[h]
  \centering \setlength{\unitlength}{1mm}
\begin{picture}(100,30)(-50,-10)

\put(-60,0){\vector(1,0){100}} %
\put(-10,-1){\line(0,1){2}} %
\put(-12,-3){$t$} %
\put(10,-1){\line(0,1){2}} %
\put(7,-5){$t+1$}

\put(-50,-1){\line(0,1){2}} %
\put(-50,4){$\bullet$} %
\put(-53,-5){$t-2$} %
\put(-50,5){\line(1,0){100}} %
\put(-57,4){$\epsilon_{t-2}$} %
\put(-42,2){$\leftarrow n_{t-2} \rightarrow$} %

\put(-30,-1){\line(0,1){2}}  %
\put(-33,-5){$t-1$} %
\put(-30,14){$\bullet$} %
\put(-30,10){\line(1,0){17}} %
\put(-37,9){$\epsilon_{t-1}$} %
\put(-30,11){$\leftarrow n_{t-1} \rightarrow$}

\put(-10,-10){\line(1,0){50}} %
\put(-10,-4){$\bullet$} %
\put(-14,-11){$\epsilon_{t}$} %
\put(15,-8){$\leftarrow n_{t} \rightarrow$}

\put(10,15){\line(1,0){40}} %
\put(10,9){$\bullet$} %
\put(4,14){$\epsilon_{t+1}$} %
\put(20,12){$\leftarrow n_{t+1} \rightarrow$}

\end{picture}

\caption{A path ($\bullet$) of the error duration process:
  $X_t = \epsilon_{t-2} + \epsilon_{t}$, $X_{t+1} = \epsilon_{t-2} +
  \epsilon_t + \epsilon_{t+1}$.}
\label{fig:parke}

\end{figure}

In Figure~\ref{fig:renewal}, a portion of a path of the
renewal-reward process is drawn; the interval $T_{k+1}$ is very
long, so that the process $X$ is constant over a long period of
time between $S_k$ and $S_{k+1}$.

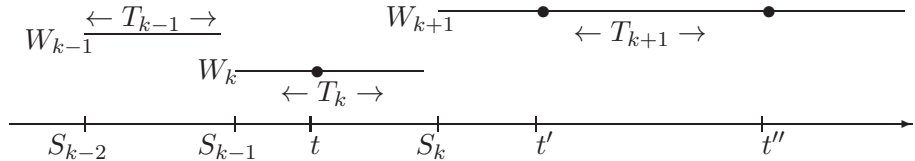
\begin{figure}[h]
  \centering \setlength{\unitlength}{1mm}
\begin{picture}(100,30)(-50,-10)

\put(-60,0){\vector(1,0){120}} %
\put(-20,-1){\line(0,1){2}}%
\put(-20,-4){$t$} %
\put(10,-1){\line(0,1){2}} %
\put(10,-4){$t'$}%
\put(40,-1){\line(0,1){2}} %
\put(40,-4){$t''$}%

\put(-50,-1){\line(0,1){2}} %
\put(-55,-4){$S_{k-2}$} %
\put(-50,12){\line(1,0){18}} %
\put(-58,10){$W_{k-1}$} %
\put(-50,13){$\leftarrow T_{k-1} \rightarrow$}

\put(-30,-1){\line(0,1){2}} %
\put(-35,-4){$S_{k-1}$}
\put(-20,6){$\bullet$} %
\put(-30,7){\line(1,0){25}} %
\put(-35,6){$W_{k}$} \put(-24,3){$\leftarrow T_{k} \rightarrow$}

\put(-3,-1){\line(0,1){2}} %
\put(-6,-4){$S_{k}$}
\put(-3,15){\line(1,0){62}} %
\put(-10,13){$W_{k+1}$} \put(15,11){$\leftarrow T_{k+1}
\rightarrow$}
\put(10,14){$\bullet$} %
\put(40,14){$\bullet$} %

\end{picture}

\caption{A path ($\bullet$) of the renewal-reward process:
  $X_t = W_{k}$, $X_{t'} = X_{t''} = W_{k+1}$.}
\label{fig:renewal}

\end{figure}


\subsection{Asymptotics For Partial Sums}
Let $X$ denote either Parke's or Taqqu-Levy's process.  The next
proposition shows that although the process $X$ is second order
stationary and its autocovariance function exhibits long range
dependence, the partial sum process of $X$ converges to a stable
L\'evy process with independent increment, which implies that its
behaviour mimics that of a sum of i.i.d. heavy tailed random
variables. In the case of the Taqqu-Levy Process, it is stated
without proof in Taqqu and Levy (1986); it can also be seen as a
particular case of Theorem 2 in Mikosch {\em et al.} (2002).
\nocite{mikosch:resnick:rootzen:stegeman:2002}
\begin{prop} \label{prop:invariance}
  Assume that (\ref{eq:regvarrenewal}) and \eqref{eq:regvar} hold with
  $1 < \alpha < 2$. Denote  $\ell(n)$ =
  $n^{-1/\alpha} \inf$ $\left\{ t>0: \; \pr( U > t) < n^{-1}
  \right\}$ with $U = T_1$ for the Taqqu-Levy process
  and $U=N$ for the Parke process. Then  the finite dimensional distributions of $\ell(n)^{-1}
n^{-1/\alpha} \sum_{k=1}^{[nt]} X_k$ converge weakly to
  those of the $\alpha$-stable Levy process $\Lambda_\alpha$ with characteristic
  function
  \begin{gather} \label{eq:cflevy}
\esp[\rme^{\rmi u \Lambda_\alpha(t)}] = \exp \left\{ - t \,
|u|^\alpha  \mu^{-1} \esp[|\xi|^\alpha] \Gamma(1-\alpha)
\cos(\pi\alpha/2) (1 -
  \rmi \, \beta \rm{sign}(u) \tan(\pi\alpha/2)) \right\},
\end{gather}
with $\beta = (\esp[\xi_+^\alpha] -
\esp[\xi_-^\alpha])/\esp[|\xi|^\alpha]$ and $\xi=W_1$ in the case
of Taqqu-Levy's process and $\xi=\epsilon_1$ and $\mu=1$ in the
case of Parke's process.
\end{prop}

\subsection{Empirical process of Taqqu-Levy's process}
In the case of Taqqu-Levy's process, the invariance principle
(i.e., the limit theorem for renormalized partial sums) can be
straightforwardly extended to an invariance principle for
instantaneous functions of the process: if $\phi$ is a measurable
function such that $\esp[\phi^2(W_1)]<\infty$ and $\esp[\phi(W_1)]
= 0$, then the finite dimensional distributions of $ \ell(n)^{-1}
n^{-1/\alpha} \sum_{k=1}^{[nt]} \phi(X_k)$ converge weakly to
those of an $\alpha$-stable Levy process, where $\ell$ is the same
slowly varying function as in proposition \ref{prop:invariance}.
For Parke's process, we conjecture that this is true for
polynomial functions.  It is actually shown in the case $\phi(x) =
x^2 - \esp[\epsilon_1^2]$ in Theorem \ref{theo:autocov}, and a
similar proof would probably work in the case of a higher order
polynomial.

In the special case of an indicator function, we obtain the usual
interval-indexed empirical process:
\[
\hat F_n(x) = \frac1n \sum_{k=1}^n \ind_{\{X_k \leq x\}}.
\]
Let  $F_W(x) = \pr(W_1 \leq x)$ be the distribution function  of
$W_1$. Then $\hat F_n$ is an estimator of $F_W$ and we have the
following  convergence result.
\begin{theo} \label{theo:empirique}
  The finite dimensional distributions of the process $\ell(n)^{-1}
  n^{1-1/\alpha}(\hat F_n - F_W)$ converges weakly to those of the
  process $\Lambda_\alpha(F(\cdot))$, where $\Lambda_\alpha$ is the
  stable Levy process with characteristic function
\[
\esp[\rme^{\rmi u \Lambda_\alpha(t)}] = \exp \left\{ - t \,
|u|^\alpha
  \mu^{-1} \Gamma(1-\alpha) \cos(\pi\alpha/2) (1 - \rmi \,
  \rm{sign}(u) \tan(\pi\alpha/2)) \right\},
\]
and $\ell(n) = n^{1/\alpha} \inf \left\{ t>0: \; \pr( T_1 > t) <
  n^{-1} \right\}$.
\end{theo}

 \section{Asymptotics for the DFTs}
 Define
\[
D_{n,j} = \sum_{t=1}^{n} X_t \, \rme^{\rmi t x_{j}},
\]
where $X$ denotes either Taqqu-Levy's or Parke's process.
\subsection{Low frequencies} \label{sec:lowfreq}
 \begin{prop} \label{prop:lowfreq}
   Define $d_{n,j}$ = $\sum_{k=1}^{[n/\mu]} \eta_k \rme^{\rmi \mu k
     x_{j}}$ with $\eta_k=T_kW_k$ for the Taqqu-Levy process and
   $\eta_k = n_k \epsilon_k$ and $\mu=1$ for Parke's process.  If
   \eqref{eq:regvarrenewal} and \eqref{eq:regvar} hold and if $j\leq
   n^\rho$ for some $\rho \in (0,1-1/\alpha)$, then $\ell(n)^{-1}
   n^{-1/\alpha} (D_{n,j}-d_{n,j}) = o_P(1)$, where
   $\ell$ is defined as in Proposition~\ref{prop:invariance}.
\end{prop}
Since in both cases $\eta_k$ belongs to the domain of attraction
of an $\alpha$-stable law, Proposition \ref{prop:lowfreq} implies
if $j\leq n^\rho$, then $\ell(n)^{-1} n^{-1/\alpha} D_{n,j}$
converges to a stable distribution. The conclusion of Proposition
\ref{prop:lowfreq} continues to hold if the upper bound $n^\rho$
is replaced by $c n^\rho$, where $c$ is a positive constant. In
the case of fixed frequencies, we can describe more precisely the
asymptotic distribution of the suitably normalized DFT
coefficients.

\begin{theo} \label{theo:lowfreq}
  Let $j_1<\dots<j_q$ be $q$ fixed positive integers. Let $\ell$ be
  defined as in Proposition \ref{prop:invariance}. Then $\ell(n)^{-1}
  n^{-1/\alpha} (D_{n,j_1},\dots,D_{n,j_q})$ converge in law to the
  complex $\alpha$-stable vector $\left(\int_0^1 \rme^{2 \rmi \pi j_1
      s} d\Lambda_\alpha(s),\dots, \int_0^1 \rme^{2 \rmi \pi j_q s}
    d\Lambda_\alpha(s)\right)$, where $\Lambda_\alpha$ is the
  $\alpha$-stable Levy process with characteristic function given
  by \eqref{eq:cflevy}.

\end{theo}
\subsection{High frequencies} \label{sec:highfreq}
In the high frequency case, the asymptotic behaviour of the
discrete Fourier transform is the same as it is for linear series.
\begin{theo}  \label{theo:highfreq}
  If let $j$ be a non decreasing sequence of integers such that
  $j/n\to0$ and $j \geq n^\rho$ for some $\rho \in (1-1/\alpha,1)$,
  then $(2\pi n f(x_{j}))^{-1/2}D_{n,j}$ is asymptotically complex
  Gaussian with independent real and imaginary parts, which are each
  zero mean Gaussian with variance $1 /2$.
\end{theo}
The conclusion of Theorem \ref{theo:highfreq} continues to hold if
the lower bound $n^\rho$ is replaced by $c n^\rho$, where $c$ is a
positive constant.

\section{Asymptotics for the Sample ACF}
The empirical autocovariance is often used as a diagnostic for
long memory, hence it is of importance to investigate its
properties in the present context. For $n\geq 1$ and $k\geq0$,
define $\bar X_n = n^{-1} \sum_{k=1}^n X_k$ and
\begin{align}
  \label{eq:defempiricalautocovariance}
  \hat \gamma_n(k) = n^{-1} \sum_{t=1}^{n-k} (X_t -\bar X_n) (X_{t+k}
  - \bar X_n).
\end{align}
Since in both cases, $X$ is a second order stationary process,
$\hat \gamma_n(k)$ is an asymptotically unbiased estimator of
$\gamma(k) = \cov(X_0,X_k)$.  In the next proposition, we show
that it is also a consistent estimator and obtain its rate of
convergence and asymptotic distribution.
\begin{theo} \label{theo:autocov}
  Assume that  \eqref{eq:regvarrenewal} and \eqref{eq:regvar}
  hold,  $\esp[|\epsilon_0|^q]<\infty$ and  $\esp[|W_0|^q]<\infty$ for
  some   $q>2\alpha$. Denote $\xi_s = W_s^2 \{T_s - \esp[T_1]\}$
  or $\xi_s = \epsilon_s^2 \{n_s - \esp[n_1]\}$. Let $\mu=1$ in the
  case of Parke's process.
Then for any fixed $k\geq 0$ and any slowly varying function $h$,
\begin{gather} \label{eq:degenere}
  \hat \gamma_n(k) - \gamma(k) = \frac1n \sum_{s=1}^{[n/\mu]} \xi_s +
  o_P(h(n)n^{1/\alpha-1}).
\end{gather}
Define $\ell$ as in Proposition \ref{prop:invariance}. Then
$\ell(n)^{-1} n^{1-1/\alpha}   (\hat \gamma_n(k)-\gamma(k))$
converges weakly to an  $\alpha$-stable
  random variable $\zeta$ with characteristic function
 \[
  \esp[\rme^{ \rmi u \zeta}] =  \exp \left\{ -  \, |u|^\alpha \, m_\alpha
  \,   \Gamma(1-\alpha) \cos(\pi\alpha/2) (1 - \rmi \,   \rm{sign}(u) \tan(\pi\alpha/2))
  \right\},
  \]
  with $m_\alpha = \esp[|\epsilon_1|^{2\alpha}]$ in the case of
  Parke's process and $m_\alpha = \esp[|W_1|^{2\alpha}]/\mu$ in the case of
the Taqqu-Levy process.
\end{theo}
\begin{rem}
  The $o_P$ term in \eqref{eq:degenere} is not uniform with respect to $k$,
  but \eqref{eq:degenere} implies that for any
  fixed integers $q$,   $k_1,\dots,k_q$, the asymptotic distribution
  of the vector
  $\ell(n)^{-1} n^{1-1/\alpha}   [\hat \gamma_n(k_1)-\gamma(k_1),\dots,\hat
  \gamma_n(k_q)-\gamma(k_q)]$ is that of an $\alpha$-stable vector
  whose components are equal. Thus, the joint limiting
  distribution of a finite collection of standardized sample
  autocovariances at fixed lags is degenerate.
\end{rem}
\begin{rem}
  Since we assume that the shocks $\epsilon_t$ or rewards $W_k$ have zero mean,
  the process $X$ itself has zero mean for both models. If it is
  known to the data analyst that the mean is zero then no mean
  correction is needed in the sample autocovariances. However, in
  practice this knowledge is rarely assumed, so we have presented our
  results for the mean corrected version.
\end{rem}
\begin{rem}
The conclusions of Theorem \ref{theo:autocov} hold also for
$\ell(n)^{-1}
  n^{1-1/\alpha} (\hat\rho_n(k) - \rho(k))$ where $\hat \rho_n(k)$
  and $\rho(k)$ are the sample and population autocorrelations at
  lag $k$. This non-Gaussian limiting distribution
  (as well as the degeneracy described above) for the
  standardized sample autocorrelations will clearly affect the
  asymptotic properties of parametric method-of-moments estimators which are based
  on a finite number of sample autocorrelations. See, for example,
  Tieslau, Schmidt and Baillie (1996).
  \end{rem}

\section{Simulations}
Throughout this section, we denote the long-memory parameter by $d
\in (0,0.5)$. Note that $d=H-1/2=1-\alpha/2$. In all of our
simulations, we use a sample size of $n=10000$. We chose to use
$ARFIMA(0,d,0)$ autocovariances in our simulations because they
are nonnegative and monotone non-increasing, which is consistent
with the nonnegative and non-increasing autocovariances implied by
both the Taqqu-Levy and Parke models.  Let $\gamma(t)$ be the
autocovariance sequence of an $ARFIMA(0,d,0)$ process,
\begin{gather} \label{eq:autocov}
  \gamma(t) =
  \frac{\Gamma(t+d)\Gamma(1-2d)}{\Gamma(t-d+1)\Gamma(1-d)\Gamma(d)} \
  \sigma^2_0, \ t=0,1,\ldots
\end{gather}
where $\sigma_0$ is the standard deviation of the $ARFIMA$
innovations.  For the integer-valued interarrival time $S_0$ as
well as the $\{T_k\}$ in the Taqqu-Levy process and the survival
times $\{n_s\}$ in the Parke process, we use the following
simulation algorithm : Let $X$ denote either $S_0$, $T_k$ or $n_s$
and let $G(x)=P(X \ge x)$. We can simulate an observation $x$ of
$X$ by drawing an observation $u$ of a uniform random variable and
setting $x$ to be the integer such that
\begin{gather} \label{eq:G(x)}
  G(x)\ge u > G(x+1).
\end{gather}
In all cases we consider here, $G(x)$ is expressed in terms of the
Gamma function, so that there is an easily evaluated continuous
increasing function $\tilde{G}(x)$ which is equal to $G(x)$ for
all integer values at which $G(x)$ is defined. The solution to
(\ref{eq:G(x)}) can be written as
\begin{gather} \label{eq:inverseG}
  x=\lfloor\tilde{G}^{-1}(u)\rfloor,
\end{gather}
where $\lfloor x \rfloor$ denotes the greatest integer less than
$x$. We obtain the solution $x$ to \eqref{eq:inverseG} using a
simple bisection algorithm (see, eg,  Johnson and Riess 1977 page
115).

\subsection{Simulation of Taqqu-Levy Process}
Before describing our sampling algorithm, we provide some
convenient formulas for $P(S_0 \ge t)$ and $P(T_k \ge t)$. From
\eqref{eq:SoEq} and Proposition \ref{prop:Xstationnaire}, we have
\begin{gather*}
\mu=\frac{1}{P(S_0=0)} \mbox{ and } \sigma^2_W=\gamma(0)
\end{gather*}
and
\begin{gather} \label{eq:pdfS0(2)}
  P(S_0 \ge t) = \frac{\gamma(t)}{\gamma(0)}, \ t=0,1,2,\ldots
\end{gather}
Thus, for $t\geq 1$, we have:
\begin{align}  \nonumber
  P(T_k \geq t) & = \mu P(S_0=t-1) = \frac{P(S_0=t-1)}{P(S_0=0)} \\
  & = \frac{\pr(S_0 \geq t-1) - \pr(S_0 \geq t)}{\pr(S_0 \geq 0) -
    \pr(S_0 \geq 1)} = \frac{\gamma(t-1) - \gamma(t)}{\gamma(0) -
    \gamma(1)}. \label{eq:Tk}
\end{align}
For all of our simulations of the Taqqu-Levy process, we assume
that $\sigma^2_0=1$. From (\ref{eq:pdfS0(2)}) and (\ref{eq:Tk}),
we can sample $S_0$ and $\{T_k\}$ using the bisection algorithm.
We also simulate $iid$ normal random variables ${W_k}$ with mean
zero and variance $\sigma^2_W=\gamma(0)$, independent of $S_0$ and
$\{T_k\}$. The duration of the 0th regime is $S_0$ and the
duration of the $k$th regime is $T_k$ for $k \ge 1$. The value of
the series $X_t$ is constant at $W_k$ throughout the $k$th regime.
This yields the simulated realization $X_0, \ldots, X_{n-1}$.
Occasionally, the entire simulated realization was constant, as
there were no breaks before $n-1$.  Such realizations were
discarded.

\subsection{Simulation of Parke's Process}
By Proposition \ref{prop:WellDefined}, Parke's process is well
defined if and only if with probability one, for all $t$, there is
a finite number of shocks surviving at time $t$. This allows us to
simulate a process which is distributionally equivalent to Parke's
using only a finite sum
\begin{gather} \label{eq:Parke}
  X_t=\sum_{s=-J}^{t}g_{s,t}\epsilon_s, \ t=1,2,\ldots
\end{gather}
where $-J$ is the time index of the oldest shock that survives at
time $t=0$.  The non-negative integer-valued random variable $J$
has a probability distribution
\begin{gather} \label{eq:J}
  P(J\leq j)=\prod_{k=j+1}^{\infty}(1-p_k).
\end{gather}
In order to obtain the covariances \eqref{eq:autocov}, for $0 < d
< 1/2$, the survival probabilities are defined by (see Parke,
1999)
\begin{gather} \label{eq:pk_gamma}
  p_k = \frac{\Gamma(2-d)}{\Gamma(d)}
  \frac{\Gamma(k+d)}{\Gamma(k+2-d)}, \ k=0,1,2,\ldots
\end{gather}
For each realization of Parke's process, we start by sampling $J$
from the probability distribution determined by (\ref{eq:J})
truncated to the range $(0,1,2,\ldots, 10000)$. This was adequate
for the values of $d$ considered here, $d=0.1$ and $d=0.4$, since
the sum of the probabilities up to that truncation point is
extremely close to one in both cases. Next, we generate a sequence
of standard normal shocks $\{\epsilon_s\}^{n}_{s=-J}$ . The
innovation variance $\sigma^2_0$ of the $ARFIMA(0,d,0)$ process is
related to $\sigma^2_\epsilon$ (we have $\sigma^2_\epsilon = 1$)
by
\begin{gather} \label{eq:sigma}
  \sigma^2_0 =
  \frac{\Gamma(1-d)\Gamma(2-d)}{\Gamma(2-2d)}\sigma^2_\epsilon.
\end{gather}
Next we discuss the simulation of the $\{n_s\}$ sequence. Special
attention must be paid to the survival time $n_{-J}$ for the
oldest shock $\epsilon_{-J}$.  It is not sampled from the
probability distribution determined by $\{p_k\}$, but rather from
the conditional distribution
\begin{gather} \label{eq:conPk}
  P(N \ge i |N \ge J) = \frac{p_i}{p_J}, \ i\ge J.
\end{gather}
We apply the bisection algorithm to sample $n_{-J}$ and the other
$\{n_s\}_{s=-J+1}^{n}$ from (\ref{eq:pk_gamma}) and
(\ref{eq:conPk}). Using the values $\{n_s\}_{s=-J}^{-1}$, we
compute the "death time" for each prehistoric shock
$\{\epsilon_s\}_{s=-J}^{-1}$ . At each time $t \ge 0$, there may
be some past shocks dying, so the time series $X_t$ is generated
by adding a new shock to the previous value $X_{t-1}$ and
subtracting the sum of those shocks dying at time $t$.

\subsection{Simulation Results}
We performed Monte Carlo simulations to assess the finite sample
properties of the DFT coefficients in light of Theorems
\ref{theo:lowfreq} and \ref{theo:highfreq} for both the Taqqu-Levy
and Parke processes. We generated 500 replications of length
$n=10000$ in each case.  Recall that $ d=1-\frac{1}{2}\alpha $,
and $ 1 < \alpha < 2$. We used autocovariances corresponding to an
$ARFIMA(0,d,0)$ model as described earlier, with $d=0.1$ and
$d=0.4$. For each value of $d$, the normalized Fourier
coefficients were evaluated at frequency $x_j$ with $ j= 1, 2,
\lfloor n^{0.2}\rfloor, \lfloor n^{0.4}\rfloor, \lfloor
n^{0.6}\rfloor, \lfloor n^{0.8}\rfloor,
\lfloor\frac{n}{2}\rfloor-2, \lfloor\frac{n}{2}\rfloor-1 $. For
the Taqqu-Levy process with $d=0.4$, there were $60$ constant
realizations. We excluded these constant realizations from our
analysis, while keeping the number of realizations used at $500$.

Figures \ref{fig:qqplot1}-\ref{fig:qqplot2} present the normal
Quantile-Quantile (QQ) plots of the normalized Fourier cosine
coefficients $A_j/{f(x_j)}^{\frac{1}{2}}$ for the Parke process
with $d=0.1$ and $d=0.4$, where
\begin{gather}
  A_j = \frac{1}{(2\pi n)^{1/2}}\sum_{t=0}^{n-1}x_t\cos(x_j t).
\end{gather}
The number inside the parenthesis at the bottom of each QQ plot
represents the $p$-value for the Anderson-Darling test of
normality. According to Theorems \ref{theo:lowfreq} and
\ref{theo:highfreq}, if $j$ increases sufficiently quickly with
the sample size $n$, i.e. when $j\ge n^{\rho}$ for $\rho >
1-1/\alpha$, the normalized Fourier coefficients are
asymptotically normal. Furthermore, as $d$ increases, the value of
$\alpha$ will decrease, and the condition on the rate of increase
of $j$ to ensure asymptotic normality becomes less stringent. When
$d=0.1$, we have $1-1/\alpha=0.4444$, a number larger than
$1-1/\alpha=0.1667$ when $d=0.4$. For the Parke process with
$d=0.1$, we do not reject the hypothesis of normality for $j \ge
n^{0.4}$; while $d=0.4$, we reject the hypothesis of normality for
$j < n^{0.2}$. Thus our simulation results are essentially
consistent with the results of Theorems \ref{theo:lowfreq} and
\ref{theo:highfreq}. We found similar results for the Taqqu-Levy
process. Since the results for the normalized Fourier sine
coefficient
\begin{gather}
  B_j = \frac{1}{(2\pi n)^{1/2}}\sum_{t=0}^{n-1}x_t\sin(x_j t)
\end{gather}
are very similar to those we found here, we do not present them
here.

Figure \ref{fig:scatterplot1} presents scatterplots of the average
log normalized periodogram vs. $\log|2\sin(x_j/2)|$ at the Fourier
frequencies from $j=1, \ldots, 4999$. We would expect a horizontal
line across all frequencies if
$E\left[\log\frac{I(x_j)}{f(x_j)}\right]$ is constant for all $j$.
The plots indicate that at low Fourier frequencies, the average
log normalized periodogram is changing but approaches a constant
as $j$ increases. If $I(x_j)/f(x_j)$ were distributed as
$(1/2)\chi_2^2$ as would be the case for a Gaussian white noise
process, we would have $E\left[I(x_j)/f(x_j)
\right]=-\gamma=-0.577216$ in Figure \ref{fig:scatterplot1}. There
seems to be some evidence that the log normalized periodogram is
biased upward for the Taqqu-Levy process with $d=0.4$, but not for
the other situations considered. Note that since the DFT
coefficients converge weakly to an $\alpha$ stable law at fixed
low Fourier frequencies, we should expect higher variability of
the log normalized periodogram at these frequencies. This suggests
that if we regress $\{\log(I(x_j))\}$ on $\{\log(f(x_j))\}$
without trimming a set of low Fourier frequencies, we may get a
biased and/or highly variable GPH estimator. Further evidence is
given in Figure \ref{fig:scatterplot2}, which presents
scatterplots of the average of $\log(I(x_j))$ vs. $\log
2|\sin(x_j/2)|$ together with their fitted least-squares lines. We
also found that there are several outliers at low frequencies for
both processes with $d=0.1$ as well as $d=0.4$. However, there are
more outliers in the case of $d=0.1$ for both processes. The fact
that the normalized periodogram behaves differently at the low
Fourier frequencies may present a problem for the GPH estimator if
we include all Fourier frequencies. The presence of the low
outliers found above is not surprising since under DDLRD, for
fixed $j$, $I(x_j) /f(x_j) =o_p(1)$ for fixed $j$ under DDLRD. See
Comment 3 in Section 7.

Figure \ref{fig:qqplotacf} presents normal QQ plots for the sample
autocorrelations based on the Taqqu-Levy process with $d=0.1$. The
Anderson-Darling $p$-values are extremely small so we reject the
null hypothesis of normality in all cases. Furthermore, the plots
indicate long-tailed distributions. These findings do not
contradict Theorem \ref{theo:autocov} which states that the
autocovariances for both processes will converge to an
$\alpha$-stable law. We found similar results for the Taqqu-Levy
process with $d=0.4$ as well as the Parke process for both values
of $d$.

Tables \ref{tab:table1} and \ref{tab:table2} present simulation
variances of the normalized DFT cosine coefficients and the
corresponding normal-based $95 \%$ confidence intervals for the
true variance, $\sigma^2$. We do not reject the null hypothesis
that $\sigma^2=0.5$ for any $j$ when $d=0.1$ in the Taqqu-Levy
process, but when $d=0.4$, we reject the null hypothesis for
$j=n/2-1$. For the Parke process, we accept the null hypothesis
for all Fourier frequencies with both values of $d$ except for
$j=n^{0.2}$ in the case $d=0.1$. Thus the results are essentially
consistent with the theoretical variances stated in Theorem
\ref{theo:highfreq}.

\section{Concluding remarks; topics for future research}

\begin{enumerate}
\item The main theoretical results we have obtained for the Parke
and Taqqu-Levy models are strikingly similar. Also, it seems clear
that the class of processes having DDLRD is much larger than the
two processes we have considered in this paper. A specific example
of another such process is the random coefficient autoregression
studied in Leipus and Surgailis (2002). We have so far been unable
to find an overarching unification for DDLRD processes which would
allow the development of a single set of theoretical results that
applies to the entire class, although such a unification seems
desirable, and may well be possible.

\item In Robinson (1995a), the theory of a modified GPH estimator
was developed for Gaussian long-memory processes. One aspect of
the modification was that an increasing number of low frequencies
were trimmed (omitted) before constructing the estimate.
Subsequently Hurvich, Deo and Brodsky (1998), who also assumed
Gaussianity, showed that trimming can be avoided. More recently,
Hurvich, Moulines and Soulier (2002) showed that trimming can also
be avoided in a different log-periodogram regression estimator,
assuming a linear, potentially non-Gaussian series. For linear
series, it is known that the DFT at fixed $j$ is asymptotically
normal (Terrin and Hurvich, 1994), but that the periodogram is
asymptotically neither independent, identically distributed, nor
exponentially distributed (K\"unsch 1986, Hurvich and Beltrao
1993). Simulations, mostly from Gaussian long-memory series,
indicate that trimming yields a very modest bias reduction, while
inflating the variance of the GPH estimator substantially. (See
also Deo and Hurvich 2001, in the context of LMSV models).

In contrast, the results of the present paper indicate that if the
long memory is generated by DDLRD, then trimming of low
frequencies may in fact be desirable. The DFT at fixed $j$
converges in distribution to an infinite-variance stable
distribution, but (under a different normalization; see Comment 3
below) if $j$ is allowed to increase suitably quickly a limiting
normal distribution results. It is unclear at this moment whether
trimming is needed to establish the asymptotic normality of the
GPH estimator based on a process having DDLRD, but clearly the
failure to trim low frequencies may adversely affect the
finite-sample behavior of the GPH estimator. Paradoxically, the
larger $d$ is, the less stringent the conditions on the rate of
increase of $j$ to ensure asymptotic normality. This seems to
indicate that when $d$ is larger less trimming would be needed,
both in theory and in practice. This runs counter to the effects
studied by Hurvich and Beltrao (1993) (which concern only the
second order structure of the process) which imply that the bias
of the normalized periodogram increases as $d$ increases from
zero.

In any case, it should be stressed that we have not attempted to
derive in this paper any asymptotic properties for the GPH
estimator or any other estimator of the long memory parameter
under DDLRD. We leave this as a topic for future research.

\item There is as yet no reliable way to distinguish between
linear long memory and DDLRD on the basis of an observed data set.
Some of our theoretical results may ultimately prove helpful in
this regard, but we leave this as a topic for future research. We
note here that the low-frequency periodogram ordinates, normalized
by the spectral density, are asymptotically normal under linear
long memory models, but converge in probability to zero under
DDLRD. This latter result follows since from our Proposition
\ref{prop:lowfreq}, we obtain after a short calculation that for
fixed $j$, $I(x_j)/f(x_j) = O_p (2/\alpha+\alpha-3)$, so that
$I(x_j)/f(x_j)=o_p(1)$. It also follows from Proposition
\ref{prop:lowfreq} that in the DDLRD case under a different
normalization, the periodogram $I(x_j)$ is asymptotically stable
if $j$ is fixed.

\item It is known (see Chung 2002 and the references therein) that
for a long-memory process linear in martingale differences, the
autocovariances are asymptotically normal if $d<1/4$, but converge
to a non-normal, finite-variance distribution if $d \in
(1/4,1/2)$. So the asymptotics for the sample autocovariances
depend on $d$, which is an undesirable property from the point of
view of statistical inference. Davis and Mikosch (1998) have shown
that for short-memory ARCH and GARCH models, the asymptotic
properties of the sample autocorrelations are more severe, as
there is no convergence in distribution. Now, for DDLRD, the
behavior is somewhere in between the linear long memory and
ARCH/GARCH cases, since for DDLRD the sample autocorrelations do
converge in distribution for all $d$ with $0<d<1/2$, but the
limiting distribution has infinite variance, and depends on $d$.
Thus, the properties of parametric estimators of $d$ which use a
fixed number of sample autocovariances will be strongly affected
by the presence of DDLRD.

\end{enumerate}

\newpage

\begin{figure}\caption{
   QQ Plots of the Normalized Fourier Cosine
    Coefficients $A_j/{f(\omega_j)}^{\frac{1}{2}}$ for Parke process;
    n=10000, d=0.1}
\begin{center}
  \includegraphics[height=18cm] {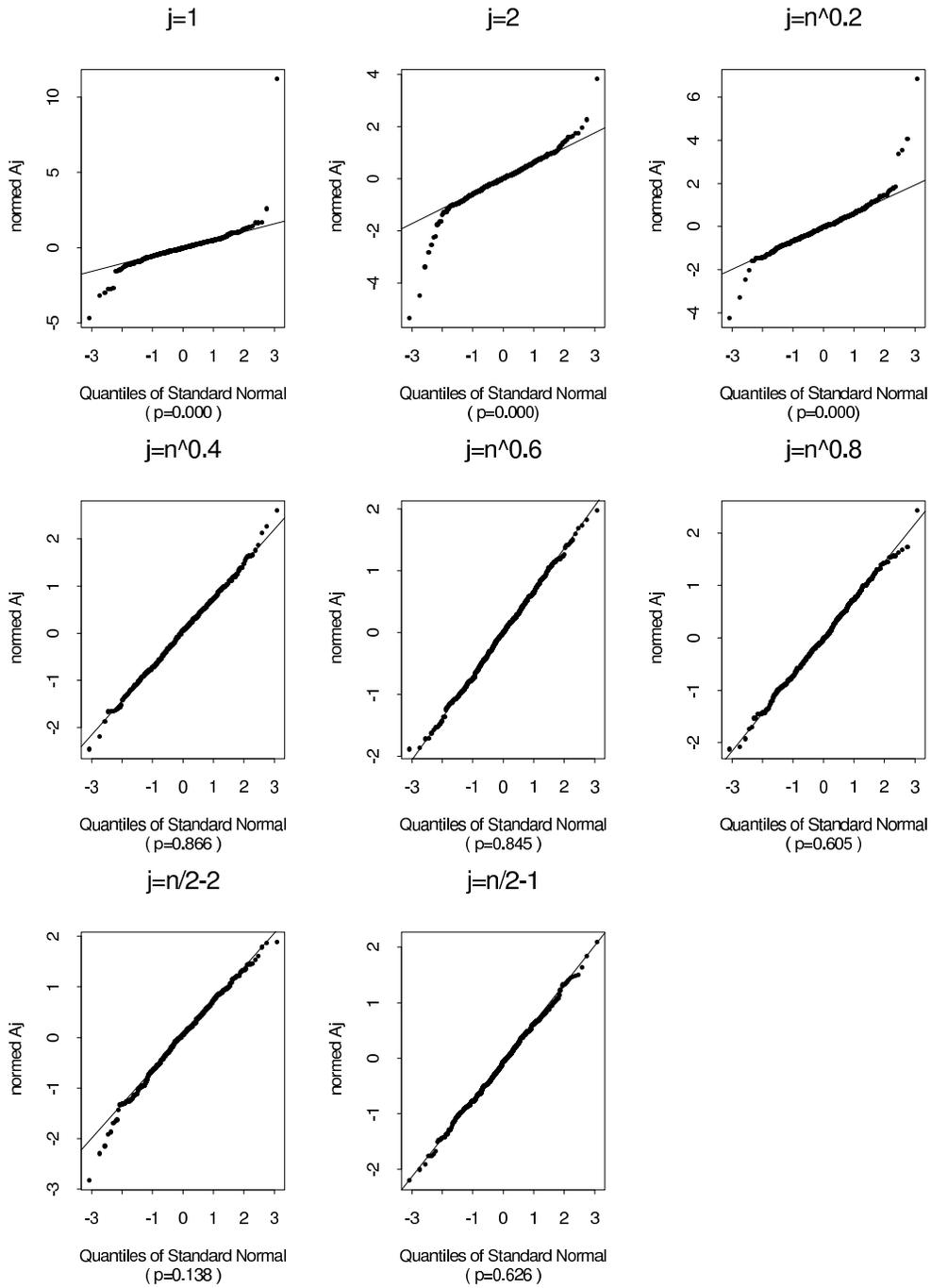}

\end{center}
 \label{fig:qqplot1}
\end{figure}
\clearpage
\begin{figure}
  \caption{QQ Plots of the Normalized Fourier Cosine
    Coefficients $A_j/{f(\omega_j)}^{\frac{1}{2}}$ for Parke process;
    n=10000, d=0.4}
    \label{fig:qqplot2}
\begin{center}
  \includegraphics[height=18cm] {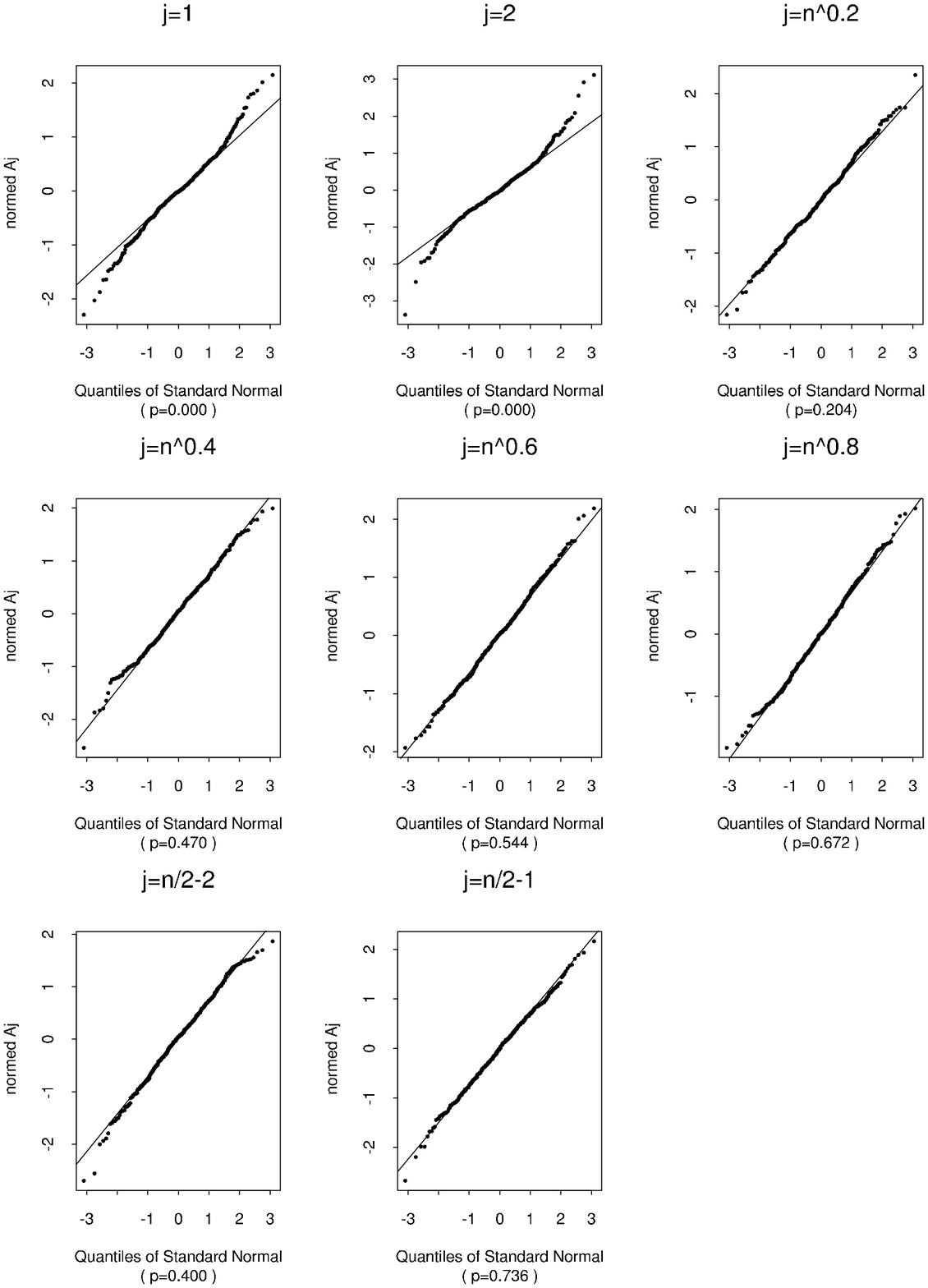}
\end{center}
\end{figure}
\clearpage
\begin{figure}
\caption{Scatterplots of Average Log Normalized
    Periodogram vs. $\log|2\sin(x_j/2)|$; j=1,2,...,4999 . Horizontal
    line represents $-\gamma=-0.577216$ .}
    \label{fig:scatterplot1}
\begin{center}
  \includegraphics[height=18cm] {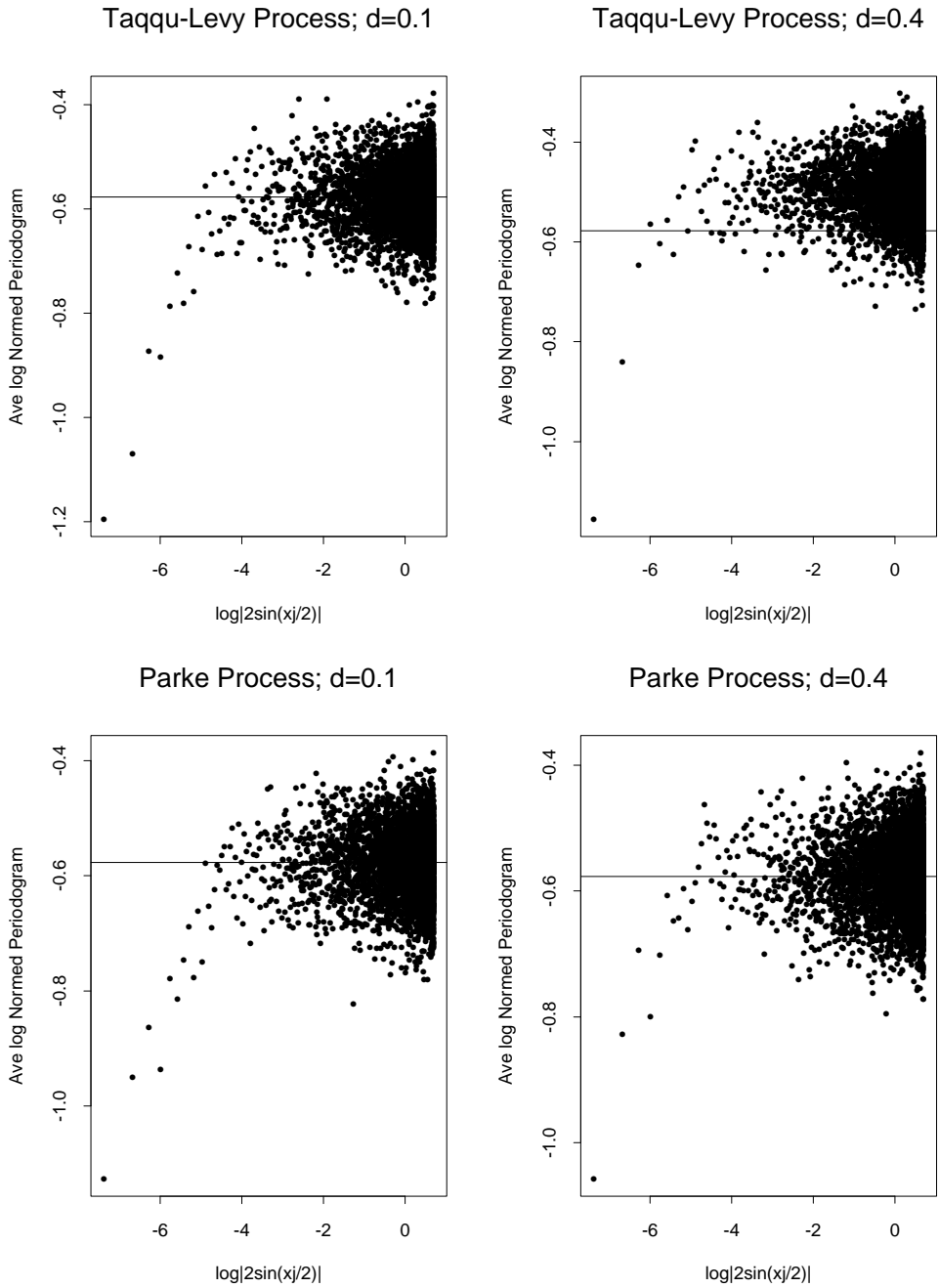}
\end{center}
\end{figure}
\clearpage
\begin{figure}
  [ptb]
\begin{center}
  \caption{Scatterplots of Average Log Periodogram vs.
    $\log|2\sin(x_j/2)|$; j=1,2,...,4999}
\label{fig:scatterplot2}
     \includegraphics[
  natheight=10.9324in, natwidth=8.6588in, height=7.27in, width=6.73in]
  {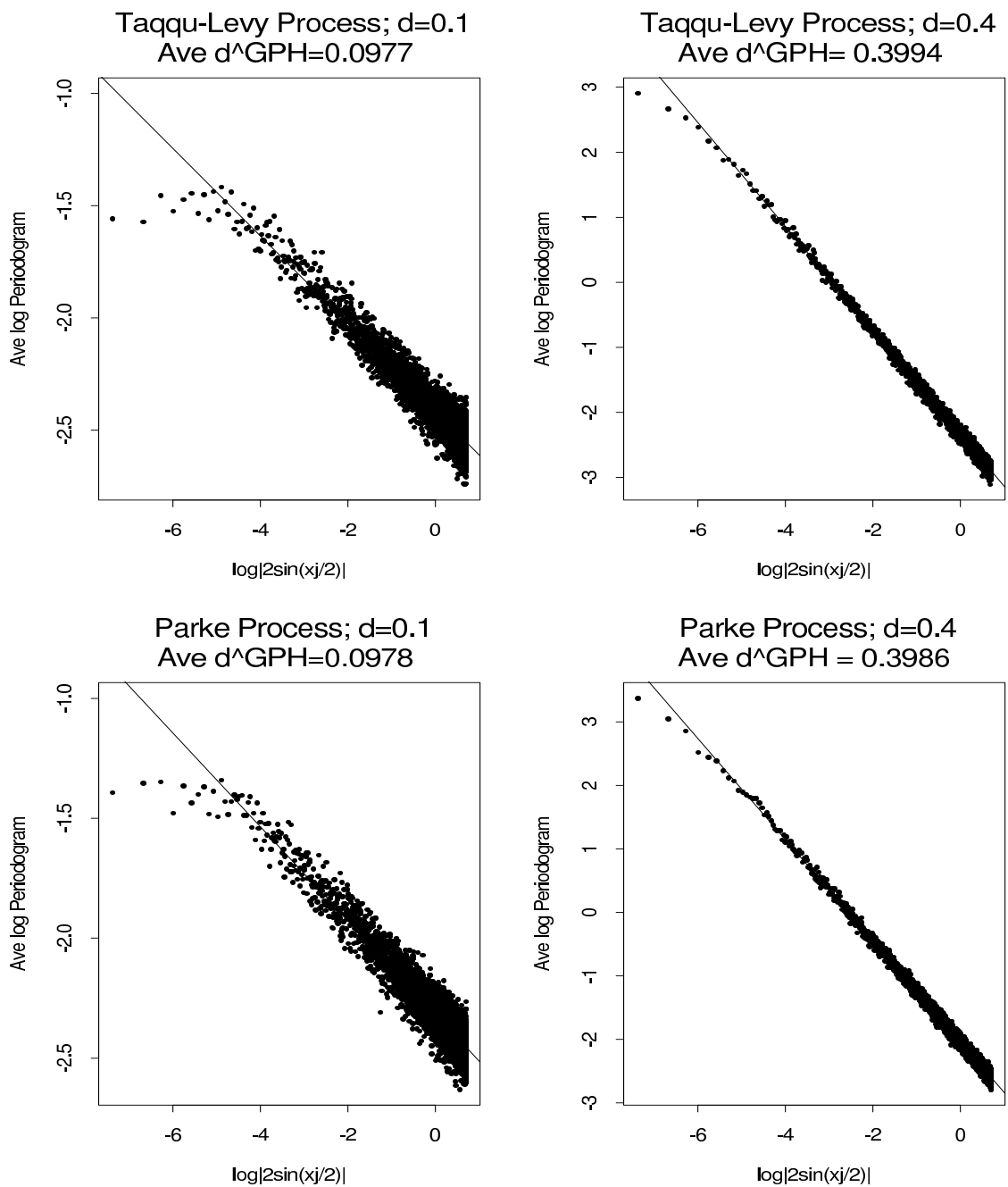}
\end{center}
 \end{figure}
\clearpage
\begin{figure}
[ptb]
\begin{center}
  \caption{Normal QQ Plots of Sample Autocorrelations for
    Taqqu-Levy Process, $d$=0.1 }
    \label{fig:qqplotacf}
    \includegraphics[height=18cm]{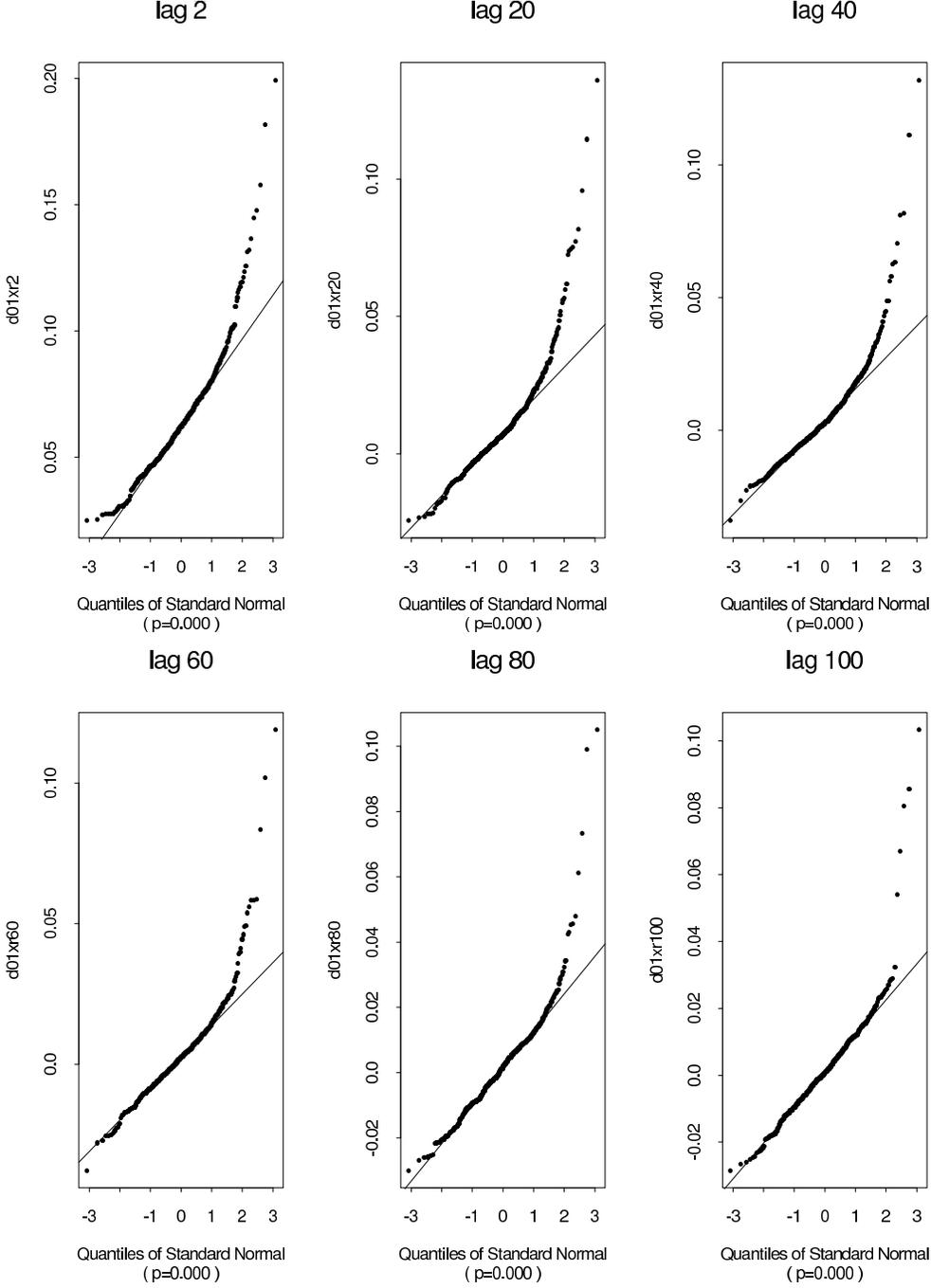}
\end{center}
\end{figure}

\clearpage
\begin{table}
\caption{Simulation variances for normalized DFT cosine
coefficients at frequency $x_j$ for Taqqu-Levy Process, with
normal-based $95\%$ Confidence Intervals. $\alpha=0.05$. Intervals
marked with $\ast$ reject the null hypothesis, $\sigma^2=0.5$ .}
\begin{center}
\label{tab:table1}
\begin{tabular}{|r|cc|c|}\hline
$d$ & $x_j$ & Variance & Confidence Interval \\[0.5ex]
\hline\hline
 0.1 &$n^{0.2}$  & 0.54  &  0.48 \ \ \ 0.62 \\
  & $n^{0.4}$  & 0.54 & 0.47 \ \ \ 0.61 \\
  & $n^{0.6}$  & 0.49 & 0.44 \ \ \ 0.56 \\
  & $n^{0.8}$  & 0.53 & 0.47 \ \ \ 0.60 \\
  & $\frac{n}{2}-2$  & 0.51 & 0.45 \ \ \ 0.58 \\
  & $\frac{n}{2}-1$  & 0.50 & 0.44 \ \ \ 0.56 \\
  \hline\hline
  0.4 &$n^{0.2}$ & 0.56 & 0.49 \ \ \ 0.63 \\
  & $n^{0.4}$  & 0.55 &  0.49 \ \ \ 0.62 \\
  & $n^{0.6}$  & 0.52 &  0.46 \ \ \ 0.59 \\
  & $n^{0.8}$  & 0.55 &  0.49 \ \ \ 0.63 \\
  & $\frac{n}{2}-2$  & 0.55 & 0.48 \ \ \ 0.62 \\
  & $\frac{n}{2}-1$  & 0.58 & \ 0.51 \ \ \ 0.66* \\
  \hline
\end{tabular}
\end{center}
\end{table}
\begin{table} \caption{Simulation variances for normalized DFT
cosine coefficients at frequency $x_j$ for Parke Process, with
normal-based $95\%$ Confidence Intervals. $\alpha=0.05$. Intervals
marked with $\ast$ reject the null hypothesis, $\sigma^2=0.5$ .}
\begin{center}
\label{tab:table2}
\begin{tabular}{|r|cc|c|}\hline
$d$ & $x_j$ & Variance & Confidence Interval \\[0.5ex]
\hline\hline
 0.1 &$n^{0.2}$ & 0.66 &  \ 0.58 \ \ \ 0.75*\\
  & $n^{0.4}$  & 0.54 & 0.48 \ \ \ 0.61\\
  & $n^{0.6}$  & 0.46 & 0.41 \ \ \ 0.53\\
  & $n^{0.8}$  & 0.51 & 0.45 \ \ \ 0.58\\
  & $\frac{n}{2}-2$  & 0.49 & 0.44 \ \ \ 0.56\\
  & $\frac{n}{2}-1$  & 0.46 & 0.41 \ \ \ 0.53\\
  \hline\hline
  0.4 &$n^{0.2}$ & 0.47 & 0.42 \ \ \ 0.54\\
  & $n^{0.4}$  & 0.49 & 0.44 \ \ \ 0.56\\
  & $n^{0.6}$  & 0.46 & 0.41 \ \ \ 0.53\\
  & $n^{0.8}$  & 0.47 & 0.42 \ \ \ 0.53\\
  & $\frac{n}{2}-2$  & 0.54 & 0.48 \ \ \ 0.62\\
  & $\frac{n}{2}-1$  & 0.52 & 0.46 \ \ \ 0.59\\
   \hline
\end{tabular}
\end{center}
\end{table}
\clearpage

\begin{center}
{\bf   APPENDIX}
\end{center}

\appendix

\section{Lemmas}
We present some lemmas in this section. Most of these are
presumably known, but we were unable to find references for them
under the conditions we needed for our main results. We therefore
include proofs for the sake of completeness.
\begin{lem} \label{lem:remplace} Let
$(\zeta_k)_{k\in\Nset^*}$ be a martingale difference sequence such
that $\sup_{k\geq1}\esp[|\zeta_k|^p]$ $< \infty$  for all
$p<\alpha$. Then, for any slowly varying function $h$,
\[
\sum_{k=1}^{M_n} \zeta_k - \sum_{k=1}^{[n/\mu]} \zeta_k  =
o_P(h(n) n^{1/\alpha}).
\]
\end{lem}
\begin{proof}
To simplify the notation, without loss of generality, we can
assume that $\mu=1$. For all $m$, denote $S_m = \sum_{k=1}^m
\zeta_k$. By Theorem 2.5.15 in Embrechts {\em et al.} (1997),
there exists a slowly varying function $\ell$ such that
$\ell(n)^{-1} n^{-1/\alpha} (M_n-n)$ converges in distribution to
a stable law. Thus, for any sequence $\delta_n$ tending to
infinity, we have:
\begin{gather} \label{eq:cltmn}
\lim_{n\to\infty} \pr\left( |M_n-n| \geq \delta_n \ell(n)
n^{1/\alpha} \right) = 0.
\end{gather}
Let $\epsilon>0$ and $\delta_n$ be an arbitrary sequence tending
to infinity. For any slowly varying function $h$, we can write:
\begin{multline*}
  \pr\bigl(  |S_{M_n}   - S_n| \geq
  \epsilon  n^{1/\alpha} h(n) \bigr)  \\
   \leq \pr(|M_n - n| > \delta_n n^{1/\alpha} \ell(n) ) + \pr\left(|M_n-n| \leq
    \delta_n n^{1/\alpha} \ell(n);  |S_{M_n}-S_n|
    \geq  \epsilon n^{1/\alpha} h(n)\right)  \\
   \leq \pr(|M_n - n| > \delta_n n^{1/\alpha} \ell(n) )  + \pr\left(  \max_{m:
      |m-n| \leq \delta_n n^{1/\alpha} \ell(n)} |S_n - S_m| \geq \epsilon n^{1/\alpha} h(n)
  \right).
\end{multline*}
Fix some  $p\in(1,\alpha)$ and  denote $C_p =
\sup_{k\geq1}\esp[|\zeta_k|^p] < \infty$ by assumption. Denote By
Kolmogorov's and Burkholder's inequalities (cf. Hall and Heyde
(1980), Theorems 2.1 and 2.10), we obtain:
\nocite{hall:heyde:1980}
\begin{align*}
  \pr\bigl(  \max_{m: |m-n| \leq n \delta_n n^{1/\alpha} \ell(n)} &
  |S_n - S_m| \geq \epsilon n^{1/\alpha} h(n) \bigr) \\
  & \leq c \epsilon^{-1} n^{-1/\alpha} h(n)^{-1}
  \esp[|S_{n+\delta_n n^{1/\alpha}\ell(n)} -
  S_{n-\delta_n n^{1/\alpha}\ell(n)}|^p]^{1/p} \\
  & \leq c \epsilon^{-1} n^{-1/\alpha} h(n)^{-1} \left(
    \sum_{k=n-\delta_n n^{1/\alpha}\ell(n)}^{n+\delta_n n^{1/\alpha}\ell(n)} \esp[|\zeta_k|^p]
  \right)^{1/p} \\
  & \leq  c \, C_p \, \epsilon^{-1} n^{-1/\alpha} h(n)^{-1}
  (\delta_n n^{1/\alpha} \ell(n))^{1/p}.
\end{align*}
Since $p>1$, this last term is $o(1)$ is the sequence $\delta_n$
converges to infinity slowly enough.
\end{proof}

\begin{lem} \label{lem:approx1}
  Let $(\zeta_{n,k})_{1\leq k \leq n} $ be uniformly bounded random
  variables.  Let $(T_k)_{k\geq1}$ be i.i.d. random variables that
  satisfy \eqref{eq:regvar} for some $\alpha\in(1,2)$ and such that
  for all $n\geq1$ and all $k\leq n$, $T_k$ is independent of
  $\{\zeta_{n,j}, 1 \leq j < k\}$.  Let $W_k$ be i.i.d. random
  variables with zero mean and finite variance, independent of
  $\zeta_{n,k}$, $1 \leq k \leq n$ and $T_k$, $1 \leq k \leq n$.  Let
  $H$ be a bounded continuous function such that for all $u\in\Rset$
  and $v\in(0,1)$:
\begin{equation}  \label{eq:condh}
|H(u,v) - u| \leq C|u|\{u^2v^2\wedge 1 + v^2\}.
\end{equation}
If $m \leq cn$ and $j \leq n^\rho$ for some $\rho\in(0,1/\alpha)$,
then $\sum_{k=1}^m \zeta_{n,k} W_k \{H(T_k,x_{n,j}) - T_k\}$ =
$o_P(n^{1/\alpha}\ell(n))$ for any slowly varying function $\ell$.
\end{lem}
\begin{proof}
  Define $\xi_n = \sum_{k=1}^n \zeta_{n,k} \{H(jT_k/n) - jT_k/n\} W_k$
  and let $\esp_{T,\zeta}$ denote the conditional expectation with
  respect to all the variables $\zeta_{n,k}$ and $T_k$.  Since the
  variables $\zeta_{n,k}$ are uniformly bounded, and since, for
  $p\in[1,\alpha)$, the function $x \to x^{p/2}$ is concave, we
  obtain:
\begin{align*}
  \esp_{T,\zeta}\left[ \left| \xi_n \right|^{p} \right] & \leq C
  \left\{\sum_{k=1}^n \{H(T_k,x_{n,j}) -  T_k\}^2 \right\}^{p/2}
  \leq C \sum_{k=1}^{n} |H(T_k,x_{n,j}) - T_k|^p.
\end{align*}
Hence, taking expectations on both sides and applying
\eqref{eq:condh}, we obtain:
\begin{align*}
  \esp\left[ \left| \xi_n \right|^{p} \right] & \leq C \sum_{k=1}^n
  \esp[|H(T_k,x_{n,j}) - T_k|^p] \\
  & \leq C n \esp[|T_1|^p(|jT_1/n|^2 \wedge 1)^p] + C n (j/n)^{2p}
  \leq C n L(n) \{(j/n)^{\alpha-p}+(j/n)^{2p}\}.
\end{align*}
Thus, $\xi_n = O_P(\{nL(n)\}^{1/p} \{j^{\alpha/p-1}+ (j/n)^2\})$.
If $\rho < 1-1/\alpha$, then $p$ can be chosen such that
$\{nL(n)\}^{1/p}\{j^{\alpha/p-1}+ (j/n)^2\} = o(n^{1/\alpha}
\ell(n))$. Hence $\xi_n = o_P(n^{1/\alpha}\ell(n))$ for any slowly
varying function $\ell$.
\end{proof}
\begin{lem}\label{lem:approx2}
  Let $\zeta_k$ be a sequence of i.i.d. rv such that for all
  $p\in(1,\alpha)$, $\esp[|\zeta_k|^p]<\infty$, $\esp[\zeta_k]=0$ and
  $\zeta_k$ is independent of $S_0, T_1,\dots,T_{k-1}$.  Let $K$ be a
  bounded continuously differentiable function on $\Rset$, with
  bounded derivative.  Define $U_{m,n,j} = \sum_{k=1}^m
  K(S_{k-1}x_{n,j}) \zeta_k$ and $V_{m,n,j} = \sum_{k=1}^m K((k-1)\mu
  x_{n,j}) \zeta_k$. If $m \leq cn$ and $j\leq n^\rho$ for some
  $\rho\in(0,1-1/\rho$, then $U_{m,n,j}-V_{m,n,j} = o_P(n^{1/\alpha}
  \ell(n))$ for any slowly varying function $\ell$.
\end{lem}
\begin{proof}
  Denote $R_k = T_1+\cdots+T_k - k\mu$.  Since $K$ is differentiable,
  we can write:
\begin{align*}
  U_{m,n}-V_{m,n} & = \sum_{k=1}^m K'((k-1)\mu x_{n,j}+
  \varsigma_k(S_{k-1} -
  (k-1)\mu x_{n,j}) \{S_{k-1}-(k-1)\mu\}x_{n,j} \zeta_k \\
  & = x_{n,j} \sum_{k=1}^n \rho_{n,k} R_{k-1} \zeta_k + S_0 x_{n,j}
  \sum_{k=1}^n \rho_{n,k} \zeta_k,
\end{align*}
where $\rho_{n,k} = K'(\{(k-1)\mu+ \varsigma_k(S_{k-1} -
(k-1)\mu\}/n)$.  Since $\esp[|\zeta_k|] <\infty$ and $K'$ is
bounded, the last term above is trivially $O_P(1)$.  By
assumption, $\{\sum_{j=1}^k \rho_{n,k} R_{k-1} \zeta_k$, $1 \leq k
\leq n\}$ is a martingale with finite $p$-th moment for
$p<\alpha$.  Hence by the Burkholder inequality for martingales,
we have, for $p<\alpha$, $\esp[|R_k|^p]= O(k)$ and
\[
\esp \left[ \left| \sum_{k=1}^n \rho_{n,k} R_{k-1} \zeta_k
\right|^p \right] \leq C \sum_{k=1}^n \esp[|R_{k-1}|^p]  = O(n^2).
\]
Thus $x_{n,j}\sum_{k=1}^n \rho_{n,k} R_{k-1} \zeta_k = O_P(j
n^{2/p-1})$.  If $\rho < 1-1/\alpha$, then $p$ can be chosen so
that $j n^{2/p-1}=o(n^{1/\alpha} \ell(n))$ for any slowly varying
function $\ell$.
\end{proof}

\begin{lem} \label{lem:sv}
  Let $H$ be a bounded continuously differentiable function on $\Rset$ such that $H(x)=
  O(x^2)$ in a neighborhood of 0. If $T_1$ satisfies
  \eqref{eq:regvar}, then
\[
\lim_{t\to\infty} t^{\alpha} L^{-1}(t) \esp[H(T_1/t)] = \alpha
\int_0^\infty H(s) s^{-\alpha-1} ds.
\]
\end{lem}

\begin{proof}
  Assume first that $H$ has a compact support in $(0,\infty)$ and
  is continuously differentiable. Then:
 \begin{align*}
   \esp[H(T_1/t)] & = \sum_{k=1}^\infty H(k/t) \pr(T=k) =
   \sum_{k=1}^\infty H(k/t) \{\pr(T \geq k) - \pr(T \geq k-1) \} \\
   & = \sum_{k=1}^\infty \pr(T \geq k) \{ H(k/t) - H((k-1)/t) \}
   = \int_0^\infty \pr(T>\lfloor s \rfloor]+1) H'(s/t) \, d s/t   \\
   & = \int_0^\infty (\lfloor s \rfloor +1)^{-\alpha}
   L(\lfloor s \rfloor+1) H'(s/t) \, d s/t =
   \int_0^\infty (\lfloor t x \rfloor+1)^{-\alpha} L(\lfloor tx \rfloor + 1) H'(x) \, d x,
 \end{align*}
 Since $L$ is slowly varying, by Karamata's Theorem, we know that
 $\lim_{t\to\infty} L(t)^{-1} L(\lfloor tx \rfloor + 1) =1$,
 uniformly with respect to $x$
 in compact sets of $(0,\infty)$. Thus, since we have assumed that $H$
 has compact support in $(0,\infty)$, we obtain
 \begin{align*}
   \lim_{t\to\infty} t^{\alpha} L^{-1}(t) \esp[H(T_1/t)] & =
   \int_0^\infty x^{-\alpha} H'(x) \, d x = \alpha \int_0^\infty
   x^{-\alpha-1} H(x) \, d x.
 \end{align*}
 To conclude, it is sufficient to prove that
 \begin{gather} \label{eq:tight}
\lim_{A\to\infty} \limsup_{t\to\infty} t^\alpha L^{-1}(t) \esp[
H(T/t) \ind_{\{T>At \mbox{ or } T< t/A\}}] = 0.
 \end{gather}
This tightness property allows then to truncate the function $H$
and apply the first part of the proof. For any $A>0$ and $t$ large
enough, applying the assumption on the behaviour of the function
$H$ at zero, we have:
\begin{align*}
\esp[H(T/t) \ind_{\{T<t/A \}}] & = \sum_{k=1}^{t/A} H(k/t)
\pr(T=k) \leq C t^{-2} \sum_{k=1}^{t/A} k^2 \pr(T=k).
\end{align*}
Applying summation by parts and Karamata's theorem, we obtain:
\begin{align*} \sum_{k=1}^{t/A} k^2 \pr(T=k) &  = 1 +
\sum_{k=1}^{t/A} \pr(T \geq n) \{k^2 - (k-1)^2\} \leq C
A^{\alpha-2} L(At).
\end{align*}
Thus, there exists a constant $C$ such that:
\begin{align*}
\limsup_{t\to\infty} t^\alpha L^{-1}(t)  \,  \esp[H(T/t)
\ind_{\{T<t/A \}}] & \leq C A^{\alpha-2} \lim_{t\to\infty}
L(At)/L(t) = C A^{\alpha-2}.
\end{align*}
Similarly, we can show that
\begin{align*}
\limsup_{t\to\infty} t^\alpha L^{-1}(t)   \, \esp[H(T/t)
\ind_{\{T>At\}}] & \leq C A^{-\alpha} \leq C A^{\alpha-2}.
\end{align*}
This proves \eqref{eq:tight} and concludes the proof of Lemma
\ref{lem:sv}.
\end{proof}

\begin{lem} \label{lem:marche}
Let $j=j(n)$ be a sequence of integers such that $n^\beta \leq j
\leq n^\rho$ for $0 < \beta \leq \rho < 1$. Then
\begin{align}
  \pr-\lim_{n\to\infty} n^{-1} \sum_{k=1}^{n} \cos(x_j S_k) =
  \pr-\lim_{n\to\infty} n^{-1} \sum_{k=1}^{n} \sin(x_j S_k) = 0, \label{eq:cos} \\
  \pr-\lim_{n\to\infty} n^{-1} \sum_{k=1}^{n} \cos^2(x_j S_k) =
  \pr-\lim_{n\to\infty} n^{-1} \sum_{k=1}^{n} \sin^2(x_j S_k) = 1/2.
  \label{eq:coscarre}
\end{align}
where $\pr-\lim$ denotes convergence in probability.
\end{lem}
Note that \eqref{eq:coscarre} follows from \eqref{eq:cos} by the
relation $\cos^2(u) = (1+\cos(2u))/2$ and by replacing $j$ by
$2j$.

\noindent To prove Lemma \ref{lem:marche}, we use the following
theorem, which adapts Theorem 2 in Yong (1971).
\begin{theo} \label{theo:fonctioncharacteristique}
  Let $T$ be a non negative integer valued random variable in the
  domain of attraction of an $\alpha$-stable law with
  $\alpha\in(1,2)$, such that $\pr(T \geq k) = k^{-\alpha} L(k)$,
  where $L$ is slowly varying at infinity.  Let $\phi$
  be the characteristic function of $T$.  Then, for $z>0$, $\phi(z) =
  1-z^\alpha \ell_1(z) + \rmi \ell_2(z) z$ where $\ell_1$ and $\ell_2$
  are slowly varying at zero, positive in a neighborhood of zero and
  satisfy, for some finite nonzero constant $C(\alpha)$,
\[
\lim_{x\to0} \ell_1(x)/L(1/x) =  C(\alpha)
  \ \mbox{ and } \lim_{x \to 0} \ell_2(x) = \esp[T]>0.
\]
\end{theo}
We will use Theorem \ref{theo:fonctioncharacteristique} through
the following bound for the modulus of the characteristic function
of $T$:
\begin{align}  \label{eq:bornefonchar}
  |\phi(z)|^2 \leq 1 - 2 \ell(z) z^{\alpha},
\end{align}
where $\ell(z) = \ell_1(z) - \frac12 \ell_1^2 z^\alpha - \frac12
\ell_2^2(z) z^{2-\alpha}$ is slowly varying and positive in a
neighborhood of zero.
\begin{proof}[Proof of Lemma \ref{lem:marche}]
We prove that the convergence holds in $L^2$. Write
\begin{align*}
  \esp \left[ \left\{ \frac{1}{n} \sum_{k=1}^{n} \cos( x_j S_k)
    \right\}^2 \right] & = \frac1{n^2} \sum_{k=1}^n \esp[\cos^2(x_j
  S_k)] + \frac2{n^2} \sum_{k=2}^n \sum_{\ell=1}^{k-1} \esp[\cos( x_j
  S_\ell) \cos( x_j S_k)]   \\
  & = O(n^{-1}) + \frac1{n^2} \sum_{k=2}^n \sum_{\ell=1}^{k-1}
  \esp[\cos( x_j (S_\ell+S_k)) + \cos( x_j (S_k-S_\ell))], \\
  \esp \left[ \left\{ \frac{1}{n} \sum_{k=1}^{n} \sin( x_j S_k)
    \right\}^2 \right] & = O(n^{-1}) + \frac1{n^2} \sum_{k=2}^n
  \sum_{\ell=1}^{k-1} \esp[\cos( x_j (S_k-S_\ell)) - \cos( x_j
  (S_k+S_\ell))].
\end{align*}
Thus we have to show that
\begin{align}
  \lim_{n\to\infty} \frac1{n^2} \sum_{k=2}^n \sum_{\ell=1}^{k-1}
  \esp[\cos( x_j (S_k-S_\ell))] = 0 \; ,   \label{eq:toto2} \\
  \lim_{n\to\infty} \frac1{n^2} \sum_{k=2}^n \sum_{\ell=1}^{k-1}
  \esp[\cos( x_j (S_\ell+S_k))] = 0 \; . \label{eq:toto1}
\end{align}
\noindent{\em Proof of (\ref{eq:toto2}).} Applying
\eqref{eq:bornefonchar}, for large enough $n$, we have
\begin{align*}
  \left| \frac{1}{n^2} \sum_{k=2}^n \sum_{k'=1}^{k-1} \esp[ \cos
    (x_j(S_k-S_{k'})) ] \right| \leq \frac{1}{n^2} \sum_{k=2}^n
  \sum_{k'=1}^{k-1} |\phi(x_j)|^{k-k'} \leq \frac{1}{n} \sum_{k=1}^{n-1}
 \{1 - 2 \ell(x_j) x_j^\alpha \}^{k/2}.
\end{align*}
For $z\in(0,1)$ and any real number $t\geq1$, $(1-z)^ = \rme^{t
\log(1-z)} \leq \rme^{-t z}$ and
\begin{align*}
  \frac{1}{n}\sum_{k=1}^{n} (1-z)^{k/2} \leq \frac{1}{n}\sum_{k=1}^{n}
  \rme^{-k z/2} & = \frac{1- \rme^{-nz/2}}{n(\rme^{z/2}-1)} \leq
  \frac{1-e^{-nz/2}}{nz/2}.
\end{align*}
Hence:
\begin{align}
  \left| \frac{1}{n^2} \sum_{k=2}^n \sum_{k'=1}^{k-1} \esp[ \cos
    (x_j(S_k-S_{k'})) ] \right| \leq \frac{1}{n}
  \sum_{k=1}^{n-1} \{1 - 2 \ell(x_j) x_j^\alpha \}^{k/2}
  \leq \frac{1- \rme^{- \ell(x_j) n x_j^\alpha}}{n \ell(x_j)x_j^\alpha}.
  \label{eq:lastterm}
\end{align}
Under the assumption on the sequence $j$, $\lim_{n\to\infty} n
\ell(x_j) x_j^\alpha = \infty$. Thus the limit of the last term in
\eqref{eq:lastterm} is 0.  This concludes the proof of
\eqref{eq:toto2}.

\noindent{\em Proof of (\ref{eq:toto1}).} Since $S_k+S_{k'} = 2S_0
+ 2(T_1 + \cdots + T_{k'}) + T_{{k'}+1} + \cdots + T_k$, and
denoting $\phi_0$ the characteristic function of $S_0$, we have
\begin{align*}
\left|   \frac{1}{n^2} \sum_{k=2}^n \sum_{{k'}=1}^{k-1}
\esp\left[\cos
    (x_j(S_k+S_{k'})) \right] \right| & \leq  \frac{1}{n^2} \sum_{k=2}^n
  \sum_{k'=1}^{k-1}  |\phi_0(2x_j)|
  |\phi(2x_j)|^{k'} | \phi(x_j)|^{k-k'}.
\end{align*}
Applying \eqref{eq:bornefonchar}, for large enough $n$, we obtain:
\begin{align*}
  \left| \frac{1}{n^2} \sum_{k=2}^n \sum_{k'=1}^{k-1} \esp[ \cos
    (x_j(S_k+S_{k'})) ] \right| & \leq \frac{1}{n^2} \sum_{k=2}^n
  \sum_{k'=1}^{k-1} \{1 - 2 \ell(2x_j)(2x_j)^\alpha\}^{k'} \{1 -
  2 \ell(x_j) x_j^\alpha \}^{k-k'}.
\end{align*}
Since for any slowly varying function $L$ and any $\alpha>0$ the
function $z^\alpha L(z)$ is ultimately non decreasing, we obtain,
for $n$ large enough:
\begin{align*}
  \left| \frac{1}{n^2} \sum_{k=2}^n \sum_{k'=1}^{k-1} \esp[ \cos
    (x_j(S_k+S_{k'}) ] \right| & \leq \frac{1}{n} \sum_{k=2}^n \{1 - 2
  \ell(x_j) x_j^\alpha \}^{k/2}.
\end{align*}
The same line of reasoning as previously  concludes the proof of
\eqref{eq:toto1} and of Lemma \ref{lem:marche}.
\end{proof}

\section{Proof of the main results}

\begin{proof}[Proof of Proposition~\ref{prop:WellDefined}]
  A necessary and sufficient condition for this process to be well defined is
  that almost surely, for all $t\in\Zset$,
\[
\inf\{s \in \Zset, s+n_s \geq t\} > -\infty.
\]
Since the random variables $n_s$ are i.i.d. with the same
distribution as $N$, by Borel-Cantelli's Lemma this condition is
equivalent to
\[
\sum_{s \leq t} \pr(n_s \geq t - s) = \sum_{k=0}^\infty \pr(N \geq
k) < \infty.
\]
Hence the necessary and sufficient condition for Parke's error
duration process to be well defined is $\esp[N]<\infty$. The
expression of the autocovariance function is proved in Parke
(1999) Proposition 1.
\end{proof}

\begin{proof}[Proof of Proposition~\ref{prop:invariance} in the case of Parke's process]
  For any real numbers $x,y$, denote $x_+=\max(x,0)$,
  $x_-=\max(-x,0)$, $x \vee y = \max(x,y)$ and $x \wedge y=\min(x,y)$.
  Recall that $g_{s,t}=1$ if $s \leq t \leq s+n_s$ and 0 otherwise.
  Hence, we can write:
\begin{align*}
  \sum_{k=1}^n X_k & = \sum_{k=1}^n \sum_{s \leq k} g_{s,k} \epsilon_s
  = \sum_{s \leq n} \sum_{k = 1 \vee s}^{(s+n_s)_+\wedge n} g_{s,k}
  \epsilon_s \\
  & = \sum_{s \leq 0} \{(s+n_s)_+ \wedge n \} \epsilon_s +
  \sum_{s=1}^n \{(s+n_s) \wedge n - s + 1\} \epsilon_s = U_n + V_n.
\end{align*}
Since $\sum_{s\leq 0}\pr(s+n_s>0) = \sum_{k\geq 0} \pr(N > k) =
\esp[N] <\infty$, the number of terms in the sum $U=\sum_{s\leq 0}
(s+n_s)_+ \epsilon_s$ is almost surely finite. Hence $U_n$
converges almost surely to $U$ and $U_n=O_P(1)$. We now split
$V_n$ into three terms: $V_n=V_{1,n} - V_{2,n} +V_{3,n}$, with
\begin{gather*}
  V_{1,n} = \sum_{s=1}^n (n_s+1) \epsilon_s, \ \
  V_{2,n} = \sum_{s=1}^n (n_s+1) \ind_{\{s+n_s>n\}} \epsilon_s, \\
  \ \ \mbox{ and } \ \ V_{3,n} = \sum_{s=1}^n (n-s+1)
  \ind_{\{s+n_s>n\}} \epsilon_s.
\end{gather*}

Since the sequences $(n_s)$ and $(\epsilon_s)$ are i.i.d. and
independent of each other, $V_{3,n}$ has the same distribution as
$W_n = \sum_{k=1}^n k \ind_{\{n_k \geq k\}} \epsilon_k$.  Since
$\sum_{k=1}^\infty \pr(n_k \geq k) < \infty$, by Borel-Cantelli's
Lemma, almost surely there exists an integer $K$ such that for all
$k > K$, $n_k < k$. Hence $W_n$ converges almost surely to
$\sum_{k=1}^\infty k \ind_{\{n_k \geq k\}} \epsilon_k$, which is
almost surely a finite sum. This implies that $V_{3,n}=O_P(1)$.

Similarly, $V_{2,n}$ has the same distribution as $\sum_{k=1}^n
n_k \ind_{\{n_k\geq k\}} \epsilon_k$, which converges almost
surely to the almost surely finite sum $\sum_{k=1}^\infty n_k
\ind_{\{n_k\geq k\}} \epsilon_k$.  Hence $V_{2,n} = O_P(1)$.

Under assumption \eqref{eq:regvar}, $N$ is in the domain of
attraction of an $\alpha$-stable law.  Since
$\esp[\epsilon_0^2]<\infty$, by Breiman's (1965) theorem,
$(n_s+1)\epsilon_s$ is an i.i.d. sequence in the domain of
attraction of an $\alpha$-stable law.
Thus we obtain that $n^{-1/\alpha} \ell(n)^{-1} V_{1,n}$ converges
weakly to the stable distribution with characteristic function
given by \eqref{eq:cflevy} (cf. for instance Embrechts {\em et
al.} (1997), Proposition 2.2.13).
\nocite{embrechts:kluppelberg:mikosch:1997} The convergence of
finite dimensional distribution is obtained similarly.
\end{proof}

\begin{proof}[Sketch of Proof of Theorem~\ref{theo:empirique}]
  Neglecting the first and last renewal periods, write
  \begin{multline*}
    \ell(n)^{-1} n^{1-1/\alpha}\{ \hat F_n(x) - F_W(x) \} =
    \ell(n)^{-1} n^{-1/\alpha} \sum_{k=1}^{M_n} \ind_{\{W_k \leq x\}}
    (T_k - \mu) \\
    + \mu \ell(n)^{-1} n^{-1/\alpha} \sum_{k=1}^{M_n} \left\{
      \ind_{\{W_k \leq x\}} - F_W(x) \right\} + (\mu \ell(n)^{-1}
    n^{-1/\alpha} M_n - 1)F_W(x) + o_P(1).
  \end{multline*}
  By Lemma \ref{lem:remplace}, the finite dimensional distributions of $\ell(n) n^{-1/\alpha}
  \sum_{k=1}^{M_n} \ind_{\{W_k \leq x\}} (T_k - \mu)$ are
  asymptotically equivalent to those of $\ell(n) n^{-1/\alpha}
  \sum_{k=1}^{n/\mu} \ind_{\{W_k \leq x\}} (T_k - \mu)$ which converge
  to those of $\Lambda_\alpha(F(x))$. Since the variables $W_k$ are
  independent of $M_n$, we have, by the renewal theorem,
  \begin{gather*}
    \esp\left[ \left(\sum_{k=1}^{M_n} \left\{ \ind_{\{W_k \leq x\}}
          -F_W(x) \right\} \right)^2 \right] = F_W(x) \{1-F_W(x)\}
    \esp[M_n] = O(n).
  \end{gather*}
  Thus, $\sum_{k=1}^{M_n} \left\{ \ind_{\{W_k \leq x\}} -F_W(x)
  \right\} = o_P(\ell(n) n^{1/\alpha})$.
\end{proof}

\begin{proof}[Proof of Proposition \ref{prop:lowfreq} in the case of
  Parke's process]
\begin{align*}
  D_{n,j} & := \sum_{t=1}^n X_t \rme^{\rmi t x_j} = \sum_{t=1}^n
  \sum_{s \leq t} g_{s,t} \rme^{\rmi t x_j} \epsilon_s \\
  & = \sum_{s \leq 0} \sum_{t=1}^{(s+n_s)_+ \wedge n} \rme^{\rmi t
    x_j} \epsilon_s + \sum_{s=1}^n \sum_{t=s}^{(s+n_s) \wedge n}
  \rme^{\rmi t x_j} \epsilon_s = : U_{n,j} + V_{n,j}.
\end{align*}
As in the proof of Proposition \ref{prop:invariance}, the sum
defining $U_{n,j}$ is almost surely finite. If $j/n\to0$, then
$U_{n,j}$ converges almost surely to the random variable
$U=\sum_{s\leq 0} (s+n_s)_+ \epsilon_s$.  Split now $V_{n,j}$ into
three terms: $V_{n,j} = W_{n,j}-R_{n,j}+T_{n,j}$, with
\begin{gather}
  W_{n,j} = \sum_{s=1}^n \left(\sum_{t=s}^{(s+n_s)}
    \rme^{\rmi t x_j} \right) \epsilon_s,  \label{eq:defwnj} \\
  R_{n,j} = \sum_{s=1}^n \sum_{t=s}^{(s+n_s)}
  \rme^{\rmi t x_j} \ind_{\{s+n_s>n\}} \epsilon_s,  \nonumber \\
  T_{n,j} = \sum_{s=1}^n \sum_{t=s}^{n} \rme^{\rmi t x_j}
  \ind_{\{s+n_s>n\}} \epsilon_s . \nonumber
\end{gather}
Consider first $R_n$. Since the sequences $(n_s)$ and
$(\epsilon_s)$ are i.i.d. and independent of each other, we have:
\begin{align*}
  R_{n,j} & = \sum_{s=1}^n \rme^{\rmi s x_j} \frac{1-\rme^{\rmi
      (n_s+1) x_j}}{1-\rme^{\rmi x_j}} \ind_{\{s+n_s>n\}} \epsilon_s
  \stackrel{(d)}= \sum_{k=1}^n \rme^{-\rmi (k-1) x_j}
  \frac{1-\rme^{\rmi (n_k+1) x_j}}{1-\rme^{\rmi x_j}} \ind_{\{n_k \geq
    k\}} \epsilon_k,
\end{align*}
where $\stackrel{(d)}=$ denotes equality of laws.  Since almost
surely there is only a finite number of indices $k$ such that
$n_k\geq k$, if $j/n\to0$, this last sum converges almost surely
to $\sum_{k=1}^\infty (n_k+1) \ind_{\{n_k \geq k\}} \epsilon_k$.
Hence $R_{n,j}=O_P(1)$. Similarly, $T_{n,j}$ has the same
distribution as
\[
\sum_{k=1}^n \sum_{t=n-k+1}^{n} \rme^{\rmi t x_j} \ind_{\{n_k \geq
  k\}} \epsilon_k = \sum_{k=1}^n \sum_{u=0}^{k-1} \rme^{-\rmi u x_j}
\ind_{\{n_k \geq k\}} \epsilon_k.
\]
If $j/n\to0$, this last term converges to $\sum_{k=1}^\infty k
\ind_{\{n_k \geq k\}} \epsilon_k$, which is an almost surely
finite sum, whence $T_{n,j}=O_P(1)$.  In conclusion, as long as
$j/n\to0$, $W_{n,j}$ is the leading term in the decomposition of
$D_{n,j}$. Consider now $W_{n,j}$. It can be written as
\begin{align*}
  W_{n,j} & = \sum_{s=1}^n \rme^{\rmi s x_j} \frac{1-\rme^{\rmi
      (n_s+1) x_j}}{1-\rme^{\rmi x_j}} \epsilon_s = \sum_{s=1}^n
  \rme^{\rmi s x_j} \rme^{\rmi n_s x_j/2}
  \frac{\sin((n_s+1)x_j/2)}{\sin(x_j/2)}
  \; \epsilon_s \\
  & = \sum_{s=1}^n \rme^{\rmi s x_j}
  \frac{\sin((n_s+1)x_j/2)}{\sin(x_j/2)} \; \epsilon_s + \sum_{s=1}^n
  \rme^{\rmi s x_j} \left(\rme^{\rmi n_s x_j/2}-1\right)
  \frac{\sin((n_s+1)x_j/2)}{\sin(x_j/2)} \; \epsilon_s \\
  & = d_{n,j} + \sum_{s=1}^n \rme^{\rmi s x_j}
  \left(\frac{\sin((n_s+1)x_j/2)}{\sin(x_j/2)} - n_s-1\right) \;
  \epsilon_s \\
  & + \sum_{s=1}^n \rme^{\rmi s x_j} \left(\rme^{\rmi n_s
      x_j/2}-1\right) \frac{\sin((n_s+1)x_j/2)}{\sin(x_j/2)} \;
  \epsilon_s = d_{n,j} + r_{n,j}.
\end{align*}
To deal with the remainder terms, we use the following bounds:
there exists a constant $C$ such that for all $u\in\Rset$ and for
all $v\in(0,1)$,
\begin{gather*}
  \left| \rme^{\rmi u}-1 \right|  \leq C(|u| \wedge 1) \\
  \left| \frac{\sin(u v)}{\sin(v)} - u \right| \leq C |u| (|uv| \wedge
  1) + |u|v^2.
\end{gather*}
For $p\in(1,\alpha)$, applying these bounds and the moment bound
for independent zero mean random variables with finite $p$-th
moment (cf. Petrov (1995), addendum 2.6.20), we have:
\nocite{petrov:1995}
\begin{gather*}
  \esp[ |r_{n,j}|^p ] \leq C \sum_{s=1}^n \esp[(n_s)^p ((n_sj/n)
  \wedge 1)^p] = C n \esp[N^p ((Nj/n) \wedge 1)^p] + C n (j/n)^{2p}
  \esp[N^p].
\end{gather*}
Let us compute $\esp[N^p ((Nj/n)\wedge1)^p]$ for any $p>1$.
\[
\esp[N^p ((Nj/n)\wedge1)^p] = (j/n)^p\sum_{k=1}^{n/j} k^{2p} \,
\pr(N=k) + \sum_{k=n/j}^\infty k^{p} \, \pr(N=k) \leq C
(j/n)^{\alpha-p} L(n).
\]
Hence, for any $p\in(1,\alpha)$, $\esp[|r_{n,j}|]=O(L(n)
n^{1+(1-\alpha)/p} j^{\alpha/p-1})$.  If $j\leq n^\rho$ for some
$\rho\in(0,1-1/\alpha)$, then $p$ can be chosen close enough to
$\alpha$ so that $\lim_{n\to\infty} h(n) n^{-1/\alpha}
\esp[|r_{n,j}|]=0$, for any slowly varying function $h$.
\end{proof}
\begin{proof}[Proof of Proposition \ref{prop:lowfreq} in the case of
  Taqqu-Levy's process] For clarity, we denote in this proof $x_{n,j}
  = 2\pi j/n$.  By summing over each regime separately, we can express
  $D_{n,j}$ as
\begin{align*}
  D_{n,j} & = W_0 \sum_{t=0}^{S_0-1} \rme^{i t x_{n,j}} +
  \sum_{k=1}^{M_n} W_k \sum_{t=S_{k-1}}^{S_k-1} \rme^{\rmi t x_{n,j}}
  +  W_{M_n+1}   \sum_{t=S_{M_n}}^n \rme^{\rmi t x_{n,j}} \\
  & = r_{1,n,j} + w_{M_n,n,j} + r_{2,n,j},
\end{align*}
where we have defined:
\begin{align*}
  r_{1,n,j} & = W_0 \exp\{\rmi (S_0-1)x_{n,j}/2\} \frac {\sin(S_0
    x_{n,j}/2)} {\sin(x_{n,j}/2)}, \\
  r_{2,n,j} & = W_{M_n+1} \exp\{\rmi \{S_{M_n}+ (n-S_{M_n})/2\}
  x_{n,j}\}  \frac{\sin(\{n-S_{M_n}+1\}x_{n,j}/2)} {\sin(x_{n,j}/2)}, \\
  w_{m,n,j} & = \sum_{k=1}^{m} W_k \sum_{t=S_{k-1}}^{S_k-1} \rme^{\rmi
    t x_{n,j}} = \sum_{k=1}^{m} \rme^{\rmi
    \{S_{k-1}+\frac{T_k-1}2\}x_{n,j}} \frac{\sin(T_k
    x_{n,j}/2)}{\sin(x_{n,j}/2)} W_k.
\end{align*}
Obviously, $|r_{1,n,j}|\leq |W_0|S_0$, hence $r_{1,n,j} = O_P(1)$,
uniformly with respect to $j\leq n/2$.  To deal with $r_{2,n,j}$,
note that $n-S_{M_n}$ is the forward recurrence time of the
stationary renewal process $(S_n)_{n\geq0}$, hence its marginal
distribution is constant and is equal to that of $S_0$ (cf.
Resnick (1992), Theorem 3.9.1).  \nocite{resnick:1992} Thus, for
$q<\alpha-1$, $\esp[|r_{2,n,j}|^q] \leq
\esp[|W_{0}|^{q}]\esp[S_0^q] <\infty$. $r_{2,n,j}$ is also
$O_P(1)$, uniformly with respect to $j\leq n/2$. Applying Lemma
\ref{lem:remplace}, we obtain that  $w_{M_n,n,j} - w_{[n/\mu],n,j}
= o_P(n^{-1/\alpha} h(n)$, uniformly with respect to the sequence
$j$ and for any slowly varying function $h$.

\noindent We now prove that $h(n) n^{-1/\alpha}
(w_{[n/\mu],n,j}-d_{n,j}) = o_P(1)$.  Define $\tilde w_{m,n,j} =
\sum_{k=1}^{m} \rme^{\rmi
  (S_{k-1}-1/2) x_{n,j}}$ $ T_k W_k$.  Applying Lemma
\ref{lem:approx1} with $m=[n/\mu]$, $H(u,v)$ = $\rme^{\rmi u v/2}
\frac{\sin(u v/2)}{\sin(v/2)}$ and $\zeta_{n,k} = \rme^{\rmi
  (S_{k-1}-1/2) x_{n,j}}$, we obtain:
\begin{align}
  w_{[n/\mu],n,j} - \tilde w_{[n/\mu],n,j} = o_P(n^{1/\alpha}
  h(n)).
\end{align}
Define $\hat w_{m,n,j} = \sum_{k=1}^m \rme^{\rmi \{(k-1) \mu
-1/2\}
  x_{n,j}} T_k W_k$.  Applying Lemma \ref{lem:approx2} with $\zeta_k =
T_k W_k$, $K(u) = \rme^{\rmi u}$ yields
\begin{align}
  \tilde w_{[n/\mu],n,j} - \hat w_{[n/\mu],n,j} =
  o_P(n^{1/\alpha}h(n)).
\end{align}
Finally, we bound $\hat w_{[n/\mu],n,j}-d_{n,j}$.
\begin{align*}
  \hat w_{[n/\mu],n,j}-d_{n,j} & = \sum_{k=1}^{[n/\mu]} (\rme^{\rmi k
    \mu x_{n,j}}\rme^{-\rmi (\mu+1/2)x_{n,j}} - \rme^{\rmi
    k x_{[n/\mu],j}})  \zeta_k \\
  & = \sum_{k=1}^{[n/\mu]} \rme^{\rmi k \mu x_{n,j}}(\rme^{-\rmi
    (\mu+1/2)x_{n,j}} - 1) \zeta_k + \sum_{k=1}^{[n/\mu]} (\rme^{\rmi
    k \mu x_{n,j}} - \rme^{\rmi k x_{[n/\mu],j}})\zeta_k.
\end{align*}
Since $1/[n/\mu]-1/(n/\mu) = O(n^{-2})$ and $j \leq n^\rho$ with
$\rho < 1-1/\alpha$, we obtain:
\begin{align*}
  \esp[|\hat w_{[n/\mu],n,j}-d_{n,j}|] \leq C j/n =
  o(n^{1/\alpha}h(n)),
\end{align*}
for any slowly varying function $h$.
\end{proof}

\begin{proof}[Proof of Theorem~\ref{theo:highfreq} in the case of Parke's process]
  As seen in the proof of Proposition \ref{prop:lowfreq}, the main
  term in the decomposition of $D_{n,j}$ is $W_{n,j}$, defined in
  \eqref{eq:defwnj}. To prove convergence to a complex Gaussian law,
  we use the Wold device.  For $a,b$ $\in \Rset$, denote
\[
\xi_{n,s}(a,b) = \{a \cos((s+n_s/2)x_j)+ b \sin((s+n_s/2)x_j)\}
\frac{\sin((n_s+1)x_j/2) \epsilon_s}{\sin(x_j/2)}.
\]
Then $ \sum_{s=1}^n \xi_{n,s}(a,b) = a\re(W_{n,j})+
b\im(W_{n,j})$.  Denote $\sigma_n^2(a,b)$ = $\sum_{s=1}^n
\esp[\xi_{n,s}^2(a,b)]$. To prove that $\sigma_n^{-1}(a,b)
\sum_{s=1}^n \xi_{n,s}(a,b)$ is asymptotically Gaussian, it
suffices to prove that
\begin{align} \label{eq:liapounov}
\sum_{s=1}^n \esp[|\xi_{n,s}(a,b)|^q] =o(\sigma_n^q(a,b)),
\end{align}
for some $q>2$.  We first find an equivalent for
$\sigma_n^2(a,b)$. To simplify the notation, without loss of
generality, assume $\sigma_\epsilon^2 = 1$. We have
\begin{align*}
\sin^2&(x_j/2)    \esp[\xi_{n,s}^2]  = \esp[\{a
\cos((s+n_s/2)x_j)+ b
  \sin((s+n_s/2)x_j)\}^2 \sin^2((n_s+1)x_j/2)] \\
  & = a^2 \esp[\cos^2((s+n_s/2)x_j)\sin^2((n_s+1)x_j/2)]
  +ab \esp[\sin((2s+n_s) x_j)  \sin^2((n_s+1)x_j/2)] \\
  & + b^2 \esp[\sin^2((s+n_s/2)x_j) \sin^2((n_s+1)x_j/2)] \\
  & = \frac{a^2+b^2}2 \esp[\sin^2((n_s+1)x_j/2)] +
  \frac{a^2-b^2}2 \esp[\cos((2s+n_s)x_j)\sin^2((n_s+1)x_j/2)] \\
  & + a b\esp[\sin((2s+n_s)x_j)\sin^2((n_s+1)x_j/2)] \\
  & = \frac{a^2+b^2}2 \esp[\sin^2((n_s+1)x_j/2)] \\
  & + \left\{\frac{a^2-b^2}2 \cos(2sx_j) + ab \sin(2sx_j) \right\}
  \esp[\cos(n_s x_j)\sin^2((n_s+1)x_j/2)]  \\
  & - \left\{\frac{a^2-b^2}2 \sin(2sx_j) - ab \cos(2sx_j)
  \right\}\esp[\sin(n_s x_j) \sin^2((n_s+1)x_j/2)]  .
\end{align*}
Applying Lemma \ref{lem:sv}, we obtain that
\begin{gather}
  \lim_{n\to\infty} x_j^{-\alpha} L(1/x_j)^{-1} \esp[h(n_s
  x_j)\sin^2((n_s+1)x_j/2)] = \alpha \int_0^\infty h(t) \sin^2(t/2)
  t^{-\alpha-1} dt,
\end{gather}
with either $h(t) = \cos(t)$, $h(t) = \sin(t)$ or $h(t) \equiv 1$.
Now, since $j\to\infty$, we have:
\begin{align*}
\left| \frac 1 n \sum_{s=1}^n   \rme^{2 \rmi  s x_j} \right| \leq
\frac2{n|\rme^{2 \rmi x_j}-1|} = O(j^{-1}) = o(1).
\end{align*}
Thus,
\begin{multline*}
  \lim_{n\to\infty } n^{-1} x_j^{2-\alpha} L(1/x_j)^{-1} \sum_{s=1}^n
  \esp[\xi_{n,s}^2(a,b)] \\
  = 2 \alpha (a^2+b^2) \int_0^\infty \sin^2(t/2) t^{-\alpha-1} dt
  = (a^2+b^2) \; \int_0^\infty \sin(t) t^{-\alpha} dt \\
  = (a^2+b^2) \Gamma(1-\alpha) \sin(\pi(\alpha-1)/2) = (a^2+b^2)
  \frac{\Gamma(2H-1)}{2-2H} \sin(\pi H).
\end{multline*}
Hence, applying
\eqref{eq:spectraldensity}, we obtain:
\begin{gather*}
  \lim_{n\to\infty } (2\pi n f(x_j))^{-1} \sum_{s=1}^n
  \esp[\xi_{n,s}^2(a,b)] = \frac{a^2+b^2}{2}.
\end{gather*}
Hence $\sigma_n^2(a,b) \sim c n f(x_j) \to\infty$. Moreover, for
any $q>2$, we have:
\begin{gather*}
  \esp[|\xi_{n,s}|^q]  \leq C (|a|+|b|)^q x_j^{-q} \esp[|
  \sin((n_s+1)x_j/2)|^q] = O(x_j^{\alpha-q}L(1/x_j)), \\
  \sigma_n^{-q}(a,b) \sum_{s=1}^n \esp[\xi_{n,s}^q]  = O
  \left((nx_j^{\alpha})^{1-q/2}\right).
\end{gather*}
Since we have assumed that $j \gg n^{1-1/\alpha}$ and $q>2$, we
obtain that $nx_j^\alpha \to \infty$ and $\sum_{s=1}^n$
$\esp[\xi_{n,s}^q]$ = $o(\sigma_n^{q}(a,b))$ and
\eqref{eq:liapounov} holds.
\end{proof}
\begin{proof} [Proof of Theorem~\ref{theo:highfreq} in the case of Taqqu-Levy's process]
  Start by noting that if $j\gg n^{1-1/\alpha}$, then
  $\lim_{n\to\infty} (nf(x_{n,j}))^{-1/2} n^{1/\alpha}$ = 0.  Hence,
  as in the proof of Proposition \ref{prop:lowfreq}, we obtain that
  $(2\pi n)^{-1/2} f(x_{n,j})^{-1/2} (D_{n,j}-w_{[n/\mu],n,j}) =
  o_P(1)$.  We now prove that $w_{[n/\mu],n,j}$ is asymptotically
  complex Gaussian by the Wold device.  Here again, without loss of
  generality, we assume $\sigma_W^2=1$. For arbitrary real numbers $a$
  and $b$, define $v_n^2 = 2\pi n \sin^2(x_{n,j}/2) f(x_{n,j})$ and
\begin{multline*}
  \eta_{n,k} = v_n^{-1} \{a \cos((S_{k-1}+(T_k-1)/2)x_{n,j}) + b
  \sin((S_{k-1}+(T_k-1)/2)x_{n,j}) \} \sin(T_k x_{n,j}/2) W_k  \\
  = v_n^{-1} \cos(S_{k-1}x_{n,j}) \{a \cos((T_k-1)/2)x_{n,j}) + b
  \sin((T_k-1)/2)x_{n,j})\} \sin(T_k x_{n,j}/2) W_k \\
  + v_n^{-1} \sin (S_{k-1}x_{n,j}) \{-a \sin((T_k-1)/2)x_{n,j}) + b
  \cos((T_k-1)/2)x_{n,j})\} \sin(T_k x_{n,j}/2) W_k.
\end{multline*}
Then $(2\pi n)^{-1/2}f^{-1/2}(x_{n,j})\left\{a \,
\re(w_{[n/\mu],n,j})+ b \, \im (w_{[n/\mu],n,j})\right\} =
\sum_{k=1}^{[n/\mu]} \eta_{n,k}$.
Denote
\begin{align*}
  B_1(u) & = \{a \cos(u/2) + b \sin(u/2)\} \sin(u/2), \\
  B_2(u) & = \{b \cos(u/2) - a  \sin(u/2)\} \sin(u/2), \\
  \tilde \eta_{n,k} & = v_n^{-1} \left\{\cos(S_{k-1}x_{n,j}) B_1(T_k
    x_{n,j}) + \sin(S_{k-1}x_{n,j})B_2(T_k x_{n,j}) \right\} W_k,
\end{align*}
and $\tilde w_{m,n,j} = \sum_{k=1}^m \tilde \eta_{n,k}$.  Then
\[
\sum_{k=1}^{[n/\mu]} \eta_{n,k} - \tilde \eta_{n,k} =
O_P(f(x_{n,j})^{-1/2}) = o_P(1).
\]
Define $\mcm_j = \sum_{k=1}^j \tilde \eta_{n,k}$, $1 \leq j \leq
[n/\mu]$ and $\mcf=(\mcf_k)_{k\geq1}$ with $\mcf_k =
\sigma(T_j,W_j, j \leq k)$.  Then $\{\mcm_j\}$ is an
$\mcf$-martingale and $\mcm_{[n/\mu]} = \tilde w_{[n/\mu],n,j}$.
Hence, to prove that $\tilde w_{[n/\mu],n,j}$ is asymptotically
Gaussian, we must prove the conditional Lindeberg conditions:
\begin{gather}
  \mbox{ there exists $\sigma^2>0$ such that } \sum_{k=1}^{[n/\mu]}
  \esp[\tilde\eta_{n,k}^2 \mid \mcf_{k-1}]
  \stackrel P {\longrightarrow} \sigma^2, \label{eq:lindeberg1} \\
  \mbox{ and } \forall \epsilon>0, \ \ \sum_{k=1}^{[n/\mu]}
  \esp[\tilde\eta_{n,k}^2 \ind_{\{|\tilde\eta_{n,k}|\geq \epsilon \}}
  \mid \mcf_{k-1}] \stackrel P {\longrightarrow} 0.
  \label{eq:lindeberg2}
\end{gather}
To prove \eqref{eq:lindeberg1}, note that
\begin{align*}
  \esp[\tilde\eta_{n,k}^2 \mid \mcf_{k-1}] & = v_n^{-2}
  \{\cos^2(S_{k-1}  x_{n,j}) \esp[B_1^2(T_1x_{n,j})] \\
  & + \sin(2 S_{k-1}x_{n,j}) \esp[B_1(T_1x_{n,j})B_2(T_1x_{n,j})] + \sin^2(S_{k-1}
  x_{n,j}) \esp[B_2^2(T_1x_{n,j})] \}\sigma^2_W.
\end{align*}
Applying Lemmas \ref{lem:sv} and \ref{lem:marche} and using
similar computations as in the proof of the previous case, we
obtain:
\begin{align}
  \sum_{k=1}^{[n/\mu]} \esp[\tilde\eta_{n,k}^2 \mid \mcf_{k-1}]
  \stackrel P {\longrightarrow} \frac{a^2+b^2}{2}.
\end{align}
To prove \eqref{eq:lindeberg2}, since $\esp[|W^q|]<\infty$ for
some $q>2$, it is sufficient to prove that:
\begin{align}  \label{eq:final}
   \sum_{k=1}^{[n/\mu]} \esp[|\tilde \eta_{n,k}|^q]  = o(v_n^q).
\end{align}
Since $\esp[|\tilde \eta_{n,k}|^q] \leq 2^{q-1} v_n^{-q/2}
\{|B_1(T_k x_{n,j})|^q + |B_2(T_k x_{n,j})|^q \}$ and
$\esp[|B_i(T_k x_{n,j})|^q ] = O(x_{n,j}^{2}f(x_{n,j}))$, $i=1,2$,
we obtain:
\[
\sum_{k=1}^{[n/\mu]} \esp[|\tilde \eta_{n,k}|^q] = O(n v_n^{-q}
x_{n,j}^2 f(x_{n,j})) = O(v_n^{1-q/2}) = o(1).
\]
Hence \eqref{eq:final} holds. Thus we have shown that $\{2\pi n
f(x_{n,j} \}^{-1/2} D_{n,j}$ is asymptotically equivalent to $
\{2\pi n f(x_{n,j} \}^{-1/2} w_{[n/\mu],n,j}$ which converges
weakly to a standard complex normal law.
\end{proof}

\begin{proof}[Proof of Theorem~\ref{theo:autocov}]
  Define $\bar X_{n,k} = n^{-1} \sum_{j=1}^{n-k} X_j$ and $\tilde
  X_{n,k} = n^{-1} \sum_{j=k+1}^{n} X_j$.  By
  Proposition~\ref{prop:invariance}, $\bar X_n = O_P(\ell(n)
  n^{1/\alpha - 1})$, and obviously, it also holds that $\bar X_{n,k}
  = O_P(\ell(n) n^{1/\alpha - 1})$ and $\tilde X_{n,k} = O_P(\ell(n)
  n^{1/\alpha - 1})$. Thus,
  \begin{align*}
    \hat \gamma_n(k) & = n^{-1} \sum_{j=1}^{n-k} X_j X_{j+k} - \bar X_n
    \tilde X_{n,k} - \bar X_n \bar X_{n,k} + (\bar X_n)^2  \\
    & = n^{-1} \sum_{j=1}^{n-k} X_j X_{j+k} + O_P(\ell^2(n)
    n^{2/\alpha-2}).
  \end{align*}
  Thus it is sufficient to prove~(\ref{eq:degenere}) for the
  autocovariances without mean correction. From now on, we denote
  $\hat \gamma_n(k) = n^{-1} \sum_{j=1}^{n-k} X_j X_{j+k}$ and we
  pursue the proof in each case separately.
\end{proof}

\begin{proof}[Proof of Theorem \ref{theo:autocov} in the case of
  Taqqu-Levy's process]
  \begin{align*}
    \hat \gamma_n(k) & = n^{-1} \sum_{t=1}^{n-k} W_{M_t}W_{M_{t+k}} =
    n^{-1} \sum_{j,j'=0}^\infty W_j W_{j'}\sum_{t=1}^{n-k}
    \ind_{\{M_t=j\}} \ind_{\{M_{t+k}=j'\}} \\
    & = n^{-1} \sum_{j=0}^\infty W_j^2 \sum_{t=1}^{n-k}
    \ind_{\{M_t=M_{t+k}=j\}} + n^{-1} \sum_{j \ne j'=0}^\infty W_j
    W_{j'}\sum_{t=1}^{n-k} \ind_{\{M_t=j\}} \ind_{\{M_{t+k}=j'\}} \\
    & = \tilde \gamma_n(k) + r_n.
  \end{align*}
  Consider first $r_n$. Note that the sums in $j$ and $j'$ are limited
  to $n$ since by definition, $M_t \leq t$. If $j'<j$ or $j'>k$, the
  event $\{M_t=j;M_{t+k}=j'\}$ is empty. Hence:
\begin{align*}
  \esp[r_n^2] & = \frac{\sigma_W^4}{n^2} \sum_{j=0}^\infty
  \sum_{j'=j+1}^{j+k}
  \sum_{s,t=1}^{n-k} \pr(M_s=M_t=j;M_{s+k}=M_{t+k}=j') \\
  & = \frac{\sigma_W^4}{n^2} \sum_{j=0}^\infty \sum_{j'=j+1}^{j+k}
  \sum_{t=1}^{n-k}  \pr(M_t=j;M_{t+k}=j') \\
  & + \frac{\sigma_W^4}{n^2} \sum_{j=0}^\infty \sum_{j'=j+1}^{j+k}
  \sum_{1 \leq s < t \leq n-k} \pr(M_s=M_t=j;M_{s+k}=M_{t+k}=j')
\end{align*}
For $s<t$ and $j<j'$, the set $\{M_s=M_t=j;M_{s+k}=M_{t+k}=j'\}$
is empty if $s+k \leq t$. Hence:
\begin{align*}
  \esp[r_n^2] & = \frac{\sigma_W^4}{n^2}
  \sum_{t=1}^{n-k}  \pr(M_t < M_{t+k}) \\
  & + \frac{\sigma_W^4}{n^2} \sum_{s=1}^{n-k-1}
  \sum_{s+1<t<s+k-1} \pr(M_s=M_t < M_{s+k}=M_{t+k}) = O(n^{-1}).
\end{align*}
Thus $r_n(k) = O_P(n^{-1/2})$. Consider now $\tilde \gamma_n(k)$.
%
By definition of the renewal process, $M_t = M_{t+k} = j$ if and
only if $S_{j-1} \leq t < S_j$ and $T_j \geq k$. Thus
\begin{gather*}
\tilde \gamma_n(k)  = \frac 1 n \sum_{j=1}^{M_{n-k}} W_j^2
\sum_{t=1}^{n-k} \ind_{\{M_t=M_{t+k}=j\}} = \frac 1 n
\sum_{j=1}^{M_{n-k}} W_j^2(T_j - k)\ind_{\{T_j \geq k\}}.
\end{gather*}
Define $\check \gamma_n(k) = \frac 1 n \sum_{j=1}^{[(n-k)/\mu]}
W_j^2(T_j - k)\ind_{\{T_j \geq k\}}$. By    Lemma
\ref{lem:remplace}, for any slowly varying function $h$, we have
that $\check \gamma_n(k) - \check \gamma_n(k) = o_P(n^{1-1/\alpha}
h(n))$. Note now that by definition, $\esp[ (T_1 - k)\ind_{\{T_1
\geq k\}}] = \mu \pr(S_0 \geq k)$. Thus:
\begin{align*}
\check \gamma_n(k) - \gamma(k) & = \frac 1 n
\sum_{j=1}^{[(n-k)/\mu]} W_j^2 (T_j - k)\ind_{\{T_j \geq k\}}  \\
& = \frac 1 n \sum_{j=1}^{[(n-k)/\mu]} W_j^2 \{(T_j -
k)\ind_{\{T_j \geq k\}} - \esp[ (T_1 - k)\ind_{\{T_1 \geq
k\}}]\} \\
& + \frac{\mu \pr(S_0 \geq k)}n \sum_{j=1}^{[(n-k)/\mu]} \{W_j^2 -
\sigma_W^2\} +  \gamma(k)  \{\mu \frac{[(n-k)/\mu]} n -1 \} \\
& = \frac 1 n \sum_{j=1}^{[(n-k)/\mu]} W_j^2 \{(T_j -
k)\ind_{\{T_j \geq k\}} - \esp[ (T_1 - k)\ind_{\{T_1 \geq k\}}]\}
+ O_P(n^{-1/2}) \\
& = \frac 1 n \sum_{j=1}^{[(n-k)/\mu]} W_j^2 \{T_j - \esp[T_1]\} +
O_P(n^{-1/2}).
\end{align*}
Thus we conclude that for any slowly varying function $h$,
\[
\hat \gamma_n(k) - \gamma(k) = \frac 1 n \sum_{j=1}^{[n/\mu]}
W_j^2 \{T_j - \esp[T_1]\} + o_P(n^{1/\alpha-1} h(n)). \]
 The rest of the proof is straightforward, given the other proofs
 in this paper, and is omitted to save space.
\end{proof}

\begin{proof}[Proof of Theorem~\ref{theo:autocov} in the case of Parke's process]
  \begin{align*}
    \hat \gamma_n(k) & = n^{-1} \sum_{t=1}^{n-k} \sum_{s\leq t}
    \sum_{s'\leq t+k} g_{s,t} g_{s',t+k} \epsilon_s \epsilon_{s'} =
    n^{-1} \sum_{s \leq n-k} \sum_{t=1}^{n-k} \ind_{\{s \vee 1 \leq
      t \leq (s+n_s-k) \wedge (n-k) \}} \epsilon_s^2 \\
    & + n^{-1} \sum_{\stackrel{s\leq n-k;s'\leq n}{s\ne s'}}
    \sum_{t=1}^{n-k} \ind_{\{s \vee 1 \leq t \leq (s+n_s) \wedge (n-k)
      \}} \ind_{\{(s'-k)\vee1 \leq t \leq (s'+n_{s'}-k) \wedge n \}}
    \epsilon_s    \epsilon_{s'} \\
    & = \tilde \gamma_n(k) + r_n(k).
  \end{align*}
We first consider $r_n(k)$. It is split into four terms as
follows.
\begin{align*}
  r_n(k) & = n^{-1} \sum_{\stackrel{s\leq 0;s'\leq k}{s\ne s'}}
  \{(s+n_s)_+ \wedge (s'+n_{s'}-k)_+ \wedge (n-k)\}
  \epsilon_s    \epsilon_{s'} \\
  & + n^{-1} \sum_{s\leq 0} \sum_{s'=k+1}^{n} [\{(s+n_s)_+ \wedge
  (s'+n_{s'}-k) \wedge (n-k)\} - s'+k+1]
  \epsilon_s    \epsilon_{s'} \\
  & + n^{-1} \sum_{s=1}^{n-k} \sum_{s'\leq k} [\{(s+n_s) \wedge
  (s'+n_{s'}-k)_+ \wedge (n-k)\} - s+1]
  \epsilon_s    \epsilon_{s'} \\
  & + n^{-1} \sum_{\stackrel{1\leq s \leq n-k ; k+1 \leq s' \leq n}{s
      \ne s'}} [(s+n_s)\wedge(s'+n_{s'}-k)\wedge(n-k) - s
  \vee(s'-k)]  \epsilon_s \epsilon_{s'} \\
  & = r_{1,n} + r_{2,n} + r_{3,n} + r_{4,n}.
\end{align*}
By the usual Borel Cantelli argument, $n r_{1,n}$ converges to the
almost surely finite sum $\sum_{\stackrel{s\leq 0;t\leq 0}{s\ne
t+k}}$ $\{(s+n_s)_+ \wedge (t+n_{t+k})_+ \} \epsilon_s
\epsilon_{t+k}$. Hence $r_{1,n} = O_P(n^{-1})$.
  By independence of the i.i.d. sequences $(\epsilon_s)$
and $(n_s)$, the terms $r_{2,n}$ and $r_{3,n}$ have the same
distribution. We consider for instance the former.  Let $S$ be the
set of nonpositive integers $s$ such that $s+n_s\geq0$. Then $S$
is almost surely finite.  Write $r_{2,n} = n^{-1} \sum_{s\in S}
\xi_{n,s} \epsilon_s$, with
\begin{align*}
  \xi_{n,s} =  \sum_{t=1}^{n-k} [\{(s+n_s)_+ \wedge
  (t+n_{t+k}) \wedge (n-k)\} -t+1]  \epsilon_{t+k}
\end{align*}
For each $s\in S$, we have:
\[
\lim_{n\to\infty} \xi_{n,s} = \sum_{t=1}^{s+n_s} [\{(s+n_s) \wedge
  (t+n_{t+k})\} -t+1]  \epsilon_{t+k}
\]
Since $S$ is almost surely finite, we thus obtain that
\[
\lim_{n\to\infty} n r_{2,n} =  \sum_{s\in S}
\sum_{t=1}^{s+n_s} [\{(s+n_s) \wedge (t+n_{t+k})\} -t+1]
\epsilon_{t+k}, \mbox{ almost surely.}
\]
 Hence $r_{2,n} = O_P(n^{-1})$ and similarly
$r_{3,n} = O_P(n^{-1})$.  Consider now the last term $r_{4,n}$.
\begin{align} \label{eq:r4nsquare}
  \esp[r_{4,n}^2] = \sigma_\epsilon^4 n^{-2} \sum_{\stackrel{1\leq s
      \leq n-k ; 1 \leq t \leq n-k}{s \ne t+k}}
  \esp[\{(s+n_s)\wedge(t+n_{t+k})\wedge(n-k) - s\vee t\}^2].
\end{align}
This last expectation is finite, since the term inside is at most
$n_s\wedge n_{t+k}$, and if $N'$ is an independent copy of $N$,
then $N \wedge N'$ is square integrable. Indeed, we have
\begin{align} \label{eq:square}
\pr(N\wedge N' \geq k) = \pr(N\geq k)^2 = L^2(k) k^{-2\alpha}.
\end{align}
Since $L$ is slowly varying, then so is $L^2$, and since
$\alpha\in(1,2)$, then \eqref{eq:square} implies that $N \wedge
N'$ is square integrable. Let us now compute the expectation in
the rhs of \eqref{eq:r4nsquare}. Assume $s<t\leq n-k$.
\begin{align*}
  \esp[\{(s+N) & \wedge(t+N') \wedge (n-k) - s\vee t\}^2] \\
  & = \sum_{j=t-s}^{n-k-s}\sum_{j'=0}^{n-k-t} \{(s+j)\wedge(t+j') -
  t\}^2 \pr(N=j) \pr(N'=j') \\
  & = \sum_{j=t-s}^{n-k-s}\sum_{j'=0}^{j-t+s} {j'}^2 \pr(N=j)
  \pr(N'=j') \\
  & + \sum_{j=t-s}^{n-k-s} \sum_{j'=j-t+s+1}^{n-k-t} (j-t+s)^2
  \pr(N=j) \pr(N'=j') \leq C L^2(t-s) (t-s)^{2-2\alpha}.
\end{align*}
Plugging this bound into \eqref{eq:r4nsquare}, we obtain:
\begin{align*}
  \esp[r_{4,n}(k)^2] = \left\{ \begin{array}{ll} O(\tilde L(n)
      n^{2-2\alpha} ) & \mbox{ if } \alpha \in (1,3/2], \mbox{ with
        $\tilde L$ slowly varying;} \\
      O(n^{-1}) & \mbox{ if } \alpha \in (3/2,2).
\end{array} \right.
\end{align*}
In conclusion, we have shown that  $r_{n}(k) = O_P(n^{1-\alpha})$.
\noindent Consider now $\tilde \gamma_n(k)$. Still by Borel
Cantelli arguments, we have
\begin{align*}
  \tilde \gamma_n(k) & = n^{-1} \sum_{s\leq 0}
  \{(s+n_{s}-k)_+ \wedge(n-k) \} \epsilon_s^2 \\
  & + n^{-1} \sum_{s=1}^{n-k} \{(s+n_s-k) \wedge (n-k) - s + 1\}
  \ind_{\{n_s\geq k\}} \epsilon_s^2 \\
  & = n^{-1} \sum_{s=1}^{n-k} (n_s-k + 1) \ind_{\{n_s\geq
    k\}} \epsilon_s^2 + O_P(n^{-1}).
\end{align*}
%
%
%
%
Altogether, we have
\begin{align*}
  \hat \gamma_n(k) & - \gamma(k) = n^{-1} \sum_{s=1}^{n-k} (n_s-k + 1)
  \ind_{\{n_s\geq k\}} \epsilon_s^2 - \gamma(k) + O_P(n^{1-\alpha}) \\
  & = n^{-1} \sum_{s=1}^{n-k} \left\{ (n_s-k + 1) \ind_{\{n_s\geq k\}}
    - \esp[(n_s-k + 1) \ind_{\{n_s\geq k\}}]\right\} \epsilon_s^2 \\
  & \hspace*{6cm}+ \frac{\esp[(N-1+k)\ind_{\{N \geq k\}}]}n
  \sum_{s=1}^{n-k}
  \{\epsilon_s^2-\sigma_\epsilon^2\}  + O_P(n^{1-\alpha}) \\
  & = n^{-1} \sum_{s=1}^{n-k} \left\{ (n_s-k + 1) \ind_{\{n_s\geq k\}}
    - \esp[(n_s-k + 1) \ind_{\{n_s\geq k\}}]\right\}
  \epsilon_s^2 +O_P(n^{-1/2}) + O_P(n^{1-\alpha}) \\
  & = n^{-1} \sum_{s=1}^{n} \left\{n_s - \esp[N]\right\} \epsilon_s^2
  +O_P(n^{-1/2}) + O_P(n^{1-\alpha}).
\end{align*}
Thus, if $\esp[|\epsilon_0|^q]<\infty$ for some $q>2\alpha$, then
$\ell(n)^{-1} n^{1-1/\alpha} (\hat \gamma_n(k) - \gamma(k))$
converges weakly to an $\alpha$-stable distribution.
\end{proof}

\end{document}